\documentclass[11pt,reqno]{amsart}

\usepackage[section]{NLH-paper}

\title[Regularity for a degenerate heat equation]{Uniqueness and root-Lipschitz regularity\\for a degenerate heat equation}

\author{Alexander Dunlap}
\address{AD: Department of Mathematics, Duke University, 120 Science Dr, Durham, NC 27708, USA}
\email{\tt dunlap@math.duke.edu}

\author{Cole Graham}
\address{CG: Division of Applied Mathematics, Brown University, 182 George St, Providence, RI 02906, USA}
\email{\tt cole\_graham@brown.edu}

\begin{document}
\begin{abstract}
  We consider nonnegative solutions of the quasilinear heat equation $\partial_t u = \tfrac{1}{2} u\partial_x^2 u$ in one dimension.
  Our solutions may vanish and may be unbounded.
  The equation is then degenerate, and weak solutions are generally nonunique.
  We introduce a notion of strong solution that ensures uniqueness.
  For suitable initial data, we prove a lower bound on the time for which a strong solution $u$ exists and $\sqrt{u}$ remains globally Lipschitz in space.
  In a companion paper, we show that this condition is important in the study of two-dimensional nonlinear stochastic heat equations.
\end{abstract}
\maketitle

\section{Introduction}

We study nonnegative solutions of the one-dimensional nonlinear heat equation
\begin{equation}
  \label{eq:main}
  \partial_t u = \frac{1}{2} u \partial_x^2 u, \quad u(0, \anon) = u_0
\end{equation}
for $(t,x) \in [0, \infty) \times \R.$
This equation degenerates where $u = 0$.
As a consequence, \eqref{eq:main} can behave poorly: its solutions can form singularities and need not be unique.
In this paper, we constrain the onset of such pathologies through a new solution theory and quantitative lower bounds on the blow-up time of \eqref{eq:main}.

We employ these bounds in a companion paper~\cite{DG} on 2D nonlinear stochastic heat equations.
There, the deterministic one-dimensional equation \eqref{eq:main} arises as an effective equation governing certain pointwise statistics.
In this probabilistic context, the solution $u$ of \eqref{eq:main} represents the decoupling function of a forward-backward SDE, and hence can be thought of as a variance.
The corresponding ``standard deviation'' $\sqrt{u}$ is therefore also of interest.
In particular, the condition that $\sqrt{u}$ be Lipschitz in $x$---a property we call ``root-Lipschitz''---plays a major role in our analysis in~\cite{DG}.
Here, we show that under suitable conditions on the initial data, $\sqrt{u}$ remains uniformly Lipschitz in space for a quantitative period of time.

To accommodate potential singularities, one can treat \eqref{eq:main} in a weak sense.
\begin{definition}
  Given $T > 0$, we say $u \in \m{C}\big([0, T); W_{\loc}^{1,\infty}(\R)\big)$ is a \emph{weak solution} of \eqref{eq:main} on the time interval $[0, T)$ if $u \geq 0$ and for all $\psi \in \m{C}_\cc^\infty([0, T) \times \R)$, we have
  \begin{equation}
    \label{eq:weak}
    \int_{\R} u_0 \psi(0, \anon) + \int_{[0, T) \times \R} \left[u \partial_t \psi - \tfrac{1}{2}\partial_x(\psi u) \partial_x u\right] = 0.
  \end{equation}
\end{definition}
Given suitable growth bounds on the initial condition, the existence of weak solutions follows from a vanishing viscosity construction; see~\cite{DPL87} for details on a bounded domain.
Moreover, interior parabolic estimates ensure that solutions of \eqref{eq:main} are qualitatively smooth on the positive set $\{u > 0\}$ (Lemma~\ref{lem:positive}).
However, this regularity alone does not imply uniqueness.
In~\cite{DPL87}, the authors show that weak solutions are highly nonunique when $\{u = 0\}$ is nonempty.
The weak formulation permits the zero set of $u$ to ``invade'' the support at arbitrary speed (for another form of nonuniqueness, see~\cite{BDPU92}).
Thus, further conditions are necessary to recover even local-in-time well-posedness for \eqref{eq:main}.
We introduce one such condition.

In the following, let $\braket{x} \coloneqq (x^2 + 1)^{1/2}$ denote the Japanese bracket.
Given $\ell \in \R$, we let $\braket{x}^\ell L^\infty$ denote the weighted space of functions $f$ such that $\braket{x}^{-\ell} f \in L^\infty$.
We extend this notation to other spaces by analogy.
\begin{definition}
  Given $\bar{T} > 0$, a weak solution of \eqref{eq:main} on the time interval $[0, \bar{T})$ is a \emph{strong solution} if for all $T \in (0, \bar{T})$, we have $u \in \braket{x}^2 L^\infty([0, T] \times \R)$ and $\partial_x^2 u \in L^1\big([0, T]; L^\infty(\R)\big)$.
\end{definition}
We prove a comparison principle for strong solutions (Proposition~\ref{prop:comparison}) and conclude that they are unique:
\begin{theorem}
  \label{thm:unique}
  If $u_1$ and $u_2$ are strong solutions of \eqref{eq:main} on a common time interval with the same initial data, then $u_1 = u_2$.
\end{theorem}
\noindent
In fact, this is essentially a consequence of Proposition~1.14 in our companion paper~\cite{DG}; the probabilistic proof given there is quite different.

Strong solutions satisfy a number of other convenient qualitative properties.
For example, the zero set of solutions plays a major role in the analysis of \eqref{eq:main}.
The support of weak solutions does not grow in time~\cite{DPL87}, but isolated zeros can disappear~\cite{Ugh86}.
In contrast, we show that strong solutions exactly preserve the zero set in time (Proposition~\ref{prop:zeros}).
As a result, distinct intervals of positivity do not interact with one another (Corollary~\ref{cor:strong-cut-paste}), so the analysis of strong solutions reduces to the study of positive strong solutions on connected intervals.

Up to translation and reflection, the only such intervals are $\R$, $\R_+$, and $(0, L)$ for some $L > 0$.
Let $I$ be one of these intervals and let $d(x) \coloneqq \op{dist}(x, I^\cc)$ denote the distance to the complement $I^\cc$.
Let $\op{Lip}f$ denote the global Lipschitz constant of a function $f \colon \R \to \R$.
Given $\kappa\in(0,\infty)$ and $\gamma\in [0, 2]$, we consider initial conditions satisfying the following.
\DeclareRobustCommand{\Ibind}{I}
\DeclareRobustCommand{\kappabind}{\kappa}
\DeclareRobustCommand{\gammabind}{\gamma}
\begin{customhyp}{H$(\kappabind,\gammabind)$}
  \label{hyp:init}
  We have $u_0 \geq 0$, $u_0|_{I^\cc} \equiv 0$, and $\op{Lip} \sqrt{u_0} < \infty$.
  Moreover, there exists $K \in [1, \infty)$ such that the following hold for all $x \in I$.
  \begin{enumerate}[wide, labelindent=3pt, itemsep = 2ex, label = \textup{(\alph*)}]
  \item
    If $I = (0, L)$, then
    \begin{equation}
      \label{eq:hyp-int}
      K^{-1} d(x)^2 \leq u_0(x) \leq \kappa d(x)^2.
    \end{equation}

  \item
    If $I = \R$, then
    \begin{equation}
      \label{eq:hyp-line}
      K^{-1} \braket{x}^\gamma \leq u_0(x) \leq \min\{\kappa x^2 + K, K\braket{x}^\gamma\}.
    \end{equation}

  \item
    If $I = \R_+$, then
    \begin{equation}
      \label{eq:hyp-half}
      K^{-1}(x^2 \wedge x^\gamma) \leq u_0(x) \leq \min\{\kappa x^2, K x^\gamma\}.
    \end{equation}
  \end{enumerate}
\end{customhyp}
The regularity condition $\op{Lip}\sqrt{u_0} < \infty$ is motivated by our interest in root-Lipschitz solutions.
In terms of growth, if $I$ has boundary, then $u_0$ grows quadratically relative to $\partial I$.
If $I$ is unbounded, then $u_0$ grows like a power law at infinity.
In each case, growth is capped at the quantitative rate $\kappa x^2$ for some $\kappa \in (0, \infty)$.
Our main result states that \eqref{eq:main} admits a strong and root-Lipschitz solution at least until time $\kappa^{-1}$.
\begin{theorem}
  \label{thm:root-Lip}
  Let $\kappa \in (0, \infty)$ and $\gamma \in [0, 2]$.
  If $u_0$ satisfies Hypothesis~\textup{\ref{hyp:init}}, then there exists a strong solution $u$ of \eqref{eq:main} on $[0, \kappa^{-1})$ such that for all $T \in [0, \kappa^{-1})$,
  \begin{equation}
    \label{eq:root-Lip}
    \sup_{t \in [0, T]} \op{Lip} \sqrt{u(t, \anon)} < \infty.
  \end{equation}
\end{theorem}
\noindent
Notably, although Hypothesis~\textup{\ref{hyp:init}} includes the assumption $\op{Lip} \sqrt{u_0} < \infty$, the lower bound $\kappa^{-1}$ on the strong existence time does not depend on the Lipschitz constant of $\sqrt{u_0}$.
\begin{remark}
  If $I = \R$ and $\gamma < 2$, Hypothesis~\ref{hyp:init} does not involve $\kappa$: if $u_0$ satisfies Hypothesis~\ref{hyp:init}, then it does so for all $\kappa \in \R_+$.
Theorem~\ref{thm:root-Lip} then implies that \eqref{eq:main} admits a \emph{global-in-time} strong, root-Lipschitz solution.
\end{remark}
To prove Theorem~\ref{thm:root-Lip}, we transform \eqref{eq:main} into a uniformly parabolic equation via a change of variables depending on $I$ and $\gamma$; see Section~\ref{sec:CoV} for more details.
This transformation allows us to rule out two distinct pathologies before time~$\kappa^{-1}$: runaway growth due to mass moving in from infinity, and the loss of root-Lipschitz regularity near the zero set $\{u_0 = 0\}$.
We refer to these as ``blow-up at infinity'' and ``blow-up at zero,'' respectively.

To simplify the following discussion, let us restrict our attention to $\R_+$, on which $u$ could potentially exhibit both forms of blow-up.
To prevent both, we make essential use of the quadratic bound $u_0(x) \leq \kappa x^2$.
This is clearly relevant for blow-up at infinity---if $u_0$ grows too quickly, the large diffusion may move mass in from infinity in finite time.
Indeed, \eqref{eq:main} has an explicit solution $\frac{\kappa x^2}{1 - \kappa t}$ that blows up precisely at time $\kappa^{-1}$.
For further discussion in a probabilistic context, see Section~1.3 of~\cite{DG22}.

Blow-up at zero is more subtle.
It is far from clear that $u_0(x) \leq \kappa x^2$ is necessary to prevent blow-up at zero.
After all, we wish to maintain the conditions
\begin{equation*}
  \partial_x u(t, 0_+) \eqqcolon r(t) = 0 \And \partial_x^2 u(t, 0_+) \eqqcolon s(t) < \infty
\end{equation*}
that hold initially.
Differentiating \eqref{eq:main} and evaluating at $x = 0$ (where $u = 0$), we \emph{formally} find
\begin{equation}
  \label{eq:tricksy}
  \dot r = \frac{1}{2} sr \And \dot s = \frac{1}{2}s^2.
\end{equation}
So long as $s < \infty$, it seems that $r = 0$ and $u$ remains quadratic to leading order near $x = 0$.
Integrating \eqref{eq:tricksy} in time, it \emph{appears} that the condition $s(0) \leq 2\kappa$ will prevent blow-up at zero before time $\kappa^{-1}$.
This suggests that if $u_0$ has a suitable second derivative at the origin, then large values away from the origin shouldn't pose a problem.
As we will see, this is not the case.

In \eqref{eq:hyp-half} of Hypothesis~\ref{hyp:init}, we have assumed that
\begin{equation}
  \label{eq:sharp-quad}
  u_0(x) \leq \kappa x^2
\end{equation}
on the entire half-line.
Our formal analysis above might suggest that this is overly strict: perhaps one could merely impose \eqref{eq:sharp-quad} in neighborhoods of $x = 0$ and $\infty$ (to prevent blow-up at infinity).
This intuition is misleading.
Violations of \eqref{eq:sharp-quad} on compact subsets of $\R_+$ can destroy the quadratic behavior near $x = 0$ and thus cause blow-up at zero.
In short, the closed system of ODEs \eqref{eq:tricksy} is deceptive: mass can move from the bulk to $x = 0$ in finite time and disrupt the evolution of $r$ and~$s$.
(This disruption is always detrimental, in the sense that moving mass cannot \emph{extend} the solution lifetime beyond the first blow-up of \eqref{eq:tricksy}; see Proposition~\ref{prop:blow-up}.)
\begin{theorem}
  \label{thm:blow-up}
  For each $\kappa \in (0, \infty)$, there exists a nonnegative profile $\vartheta \in \m{C}_\cc^\infty(\R_+)$ such that if $u_0(x) = \kappa x^2 + \vartheta(x)$, then the unique strong solution $u$ of \eqref{eq:main} has maximal existence time $T_* \in (0, \kappa^{-1})$ and
  \begin{equation}
    \label{eq:not-root-Lip}
    \sup_{t \in [0, T_*)} \op{Lip} \sqrt{u(t, \anon)} = \infty.
  \end{equation}
\end{theorem}
\noindent
We emphasize that $\R_+$ denotes the open interval $(0, \infty)$, so $0 \not \in \supp \vartheta$.

At a high level, Theorem~\ref{thm:blow-up} states that mass moving from the bulk to the edge can cause a qualitative change in behavior.
This resembles the ``waiting time'' phenomenon studied in the porous medium equation (PME) and related models.
In the porous medium equation
\begin{equation}
  \label{eq:PME}
  \partial_t v = \partial_x^2(v^m) \quad \text{with } m > 1,
\end{equation}
the spatial support of $v$ generally grows in time, but may remain initially constant up to a positive ``waiting time.''
At the waiting time, mass from the bulk reaches the boundary and changes the order of vanishing, which induces the support to expand~\cite{ACK83}.

Superficially, our model \eqref{eq:main} appears to be a close cousin of the PME.
Indeed, one can alternatively write \eqref{eq:main} as
\begin{equation}
  \label{eq:div-form}
  \partial_t u = \frac{1}{4} \partial_x^2(u^2) - \frac{1}{2} (\partial_x u)^2.
\end{equation}
Thus \eqref{eq:main} is a form of the PME with nonlinear gradient absorption.
In fact, the classical waiting time phenomenon has also been observed~\cite{ZMZA15} in equations of the form
\begin{equation}
  \label{eq:PME-absorption}
  \partial_t w = \partial_x^2(w^m) - |\partial_x w|^q
\end{equation}
for $1 < q < m < 2$.

We emphasize that the exponents in \eqref{eq:div-form} lie just beyond this range.
And indeed, \eqref{eq:div-form} behaves quite differently to \eqref{eq:PME} and \eqref{eq:PME-absorption}.
Weak solutions of \eqref{eq:PME} are unique~\cite{OKC58} and the support of solutions to \eqref{eq:PME} and \eqref{eq:PME-absorption} can grow in time~\mbox{\cite{K67,ZMZA15}}.
In contrast, as noted above, weak solutions of \eqref{eq:div-form} are highly nonunique and their supports are nonincreasing in time~\cite{DPL87}.
We thus view Theorems~\ref{thm:root-Lip} and~\ref{thm:blow-up} as a distinct variation on the waiting-time theme.
For related results concerning the disappearance of isolated zeros in \eqref{eq:main}, see~\cite{BU90}.

\subsection*{Connection with a forward-backward SDE}
A solution to the PDE \eqref{eq:main} is a so-called decoupling function for a certain forward-backward stochastic differential equation (FBSDE; see, e.g.,~\cite{MWZZ15} for background).
Given $T > 0$ and a standard Brownian motion $B$, consider the problem
\begin{equation}
  \label{eq:FBSDE}
  \begin{aligned}
    \ds X(t) &= \sqrt{Y(t)}\ds B(t), \quad &X(0) &= x_0,\\
    \ds Y(t) &= Z(t)\ds B(t), &Y(T) &= u\big(0,X(T)\big)
  \end{aligned}
\end{equation}
posed on a time interval $[0, T]$.
This FBSDE can be alternatively written as
\begin{equation}
  \label{eq:ourusualformofFBSDE}
  \ds X(t) = \sqrt{\mathbb{E}\big[u\big(0,X(T)\big) \mid X(t)\big]}\ds B(t),\qquad X(0)=x_0.
\end{equation}
We call $u$ a decoupling function because
\begin{equation}
  \label{eq:udecoupling}
    u(T-t,x) = \mathbb{E}\big[u\big(0,X(T)\big)\mid X(t)=x\big],
\end{equation}
so the problem \eqref{eq:ourusualformofFBSDE} can be rewritten as \eqref{eq:udecoupling} together with the forward SDE
\begin{equation}
  \label{eq:Xwithdecoupling}
    \ds X(t) = \sqrt{u\big(T-t,X(t)\big)}\ds B(t), \quad X(0)=x_0.
\end{equation}
In the problem \eqref{eq:udecoupling}--\eqref{eq:Xwithdecoupling}, we must solve for $u$ as well as $X$.
However, it can be shown using Itô's formula that $u$ solves the PDE \eqref{eq:main}.
Therefore, one path to understand the FBSDE is to first study the PDE \eqref{eq:main} and then the SDE \eqref{eq:Xwithdecoupling}.
This connection between quasilinear PDEs and forward-backward SDEs is well-known; see, e.g.,~\cite[Section~8.2]{MY99}.
We prove this relationship in our setting in \cite[Proposition~1.14]{DG}; by Proposition~\ref{prop:C1} below, strong solutions satisfy the hypotheses in \cite{DG}.

Our companion article \cite{DG} proves local-in-time well-posedness for the FBSDE.
Moreover, this well-posedness holds so long as $u$ remains root-Lipschitz.
Thus Theorem~\ref{thm:root-Lip} implies that the FBSDE is well-posed for $T < \kappa^{-1}$.
PDE regularity theory has been used before to establish well-posedness for FBSDEs~\cite{MPY94,PT99,Del02}.
However, we are not aware of previous work that obtains precise estimates for degenerate problems with growth at infinity.

\subsection*{Organization}
In Section~\ref{sec:qualitative}, we establish various qualitative properties of strong solutions, including uniqueness (Theorem~\ref{thm:unique}).
We construct strong, root-Lipschitz solutions in Section~\ref{sec:existence} and thereby prove Theorem~\ref{thm:root-Lip}.
In Section~\ref{sec:blow-up}, we prove Theorem~\ref{thm:blow-up} by exhibiting blow-up at zero.

\subsection*{Acknowledgments}
We thank Hongjie Dong for his assistance with the intermediate Schauder estimate in Proposition~\ref{prop:transformed}, which supersedes a more complicated earlier argument.

The authors were supported by the NSF Mathematical Sciences Postdoctoral Research Fellowship program through grants DMS-2002118 (AD) and DMS-2103383 (CG).
Much of the work was completed while AD was at NYU Courant.

\section{Qualitative properties of strong solutions}
\label{sec:qualitative}

In this section, we study the qualitative properties of solutions of \eqref{eq:main}.
After a brief discussion of weak solutions, we move to the comparison principle and uniqueness for strong solutions.
These properties allow us to construct strong solutions with a certain degree of quantitative regularity.
We then show that strong solutions preserve the zero set of their initial data, which allows us to argue that strong solution do not generally exist for all time.

\subsection*{Notation}
Throughout the paper, we write $f \lesssim g$ when $f \leq C g$ for some constant $C \in \R_+$.
This is equivalent to $f = \m{O}(g)$.
If $f \lesssim g$ and $g \lesssim f$, we write $f \asymp g$.

\subsection{Positive smoothness}
We begin by showing that weak solutions are smooth where they are positive.
By omission, this highlights the importance of understanding the zero set where \eqref{eq:main} degenerates.
\begin{lemma}
  \label{lem:positive}
  Let $u$ be a weak solution of \eqref{eq:main} on a time interval $[0, T)$.
  Then on the set $\{(t, x) \in (0, T) \times \R \mid u(t, x) > 0\}$, $u$ is qualitatively smooth and satisfies the PDE \eqref{eq:main} pointwise.
\end{lemma}
\begin{proof}
  Let $P \coloneqq \{(t, x) \in (0, T) \times \R \mid u(t, x) > 0\}$.
  Because $u$ is continuous, $P$ is open.
  Let $Q$ be a bounded open set such that $\bar{Q} \subset P$.
  Then $0 < \inf_Q u \leq \sup_Q u < \infty$ and $\norm{\partial_x u}_{L^\infty(Q)} < \infty$.
  Thus by Theorem~6.6 of~\cite{L96}, $\partial_x^2 u \in L^2(Q)$.
  (The theorem is stated for divergence-form equations; we apply it to the formulation \eqref{eq:div-form} of \eqref{eq:main}.)
  Differentiating \eqref{eq:main}, we see that $r \coloneqq \partial_x u$ weakly satisfies the divergence-form PDE
  \begin{equation*}
    \partial_t r = \frac{1}{2} \partial_x(u \partial_x r).
  \end{equation*}
  Applying \cite[Theorem~6.6]{L96} again, we see that $\partial_x^3 u = \partial_x^2 r \in L^2(Q)$.
  In particular, by Morrey's inequality, $\partial_x^2 u \in \m{C}_x^{1/2}$.
  Using the PDE, $\partial_t u \in \m{C}_x^{1/2}$, and in particular $\partial_t u$ is continuous.
  We have thus shown that $u \in \m{C}_t^1 \cap \m{C}_x^2$.
  This allows us to integrate by parts in the weak formulation \eqref{eq:weak}.
  Using an approximation of the identity for $\psi$, we see that \eqref{eq:main} is satisfied pointwise in $Q$.
  Finally, we can repeatedly differentiate \eqref{eq:main} and apply \cite[Theorem~6.6]{L96} to conclude that $u$ is in fact smooth in $Q$.
\end{proof}

\subsection{Comparison and uniqueness}
We now focus on the special properties of \emph{strong} solutions, beginning with a comparison principle.
We emphasize that merely weak solutions do not satisfy such a principle, as the nonuniqueness exhibited in~\cite{DPL87} demonstrates.
\begin{proposition}
  \label{prop:comparison}
  Given $T > 0$, suppose $u^\pm \in \m{C}_\loc\big([0, T] \times \R\big)$ satisfy $\partial_x^2 u^\pm \in L_t^1 L_x^\infty$ and $\partial_t u^\pm \in \braket{x}^2 L_t^1L_x^\infty$ as well as
  \begin{equation}
    \label{eq:super-sub}
    \partial_t u^+ \geq \frac{1}{2} u^+ \partial_x^2 u^+ \And \partial_t u^- \leq \frac{1}{2} u^- \partial_x^2 u^-\qquad \text{a.e.}
  \end{equation}
  Then if $u^-(0, \anon) \leq u^+(0, \anon)$, we have $u^- \leq u^+$.
\end{proposition}
\begin{proof}
  Because the nonnegative continuous function $u^\pm$ is locally bounded and $\partial_x^2 u^\pm \in L_t^1L_x^\infty$, we can integrate twice in space to find
  \begin{equation}
    \label{eq:super-sub-weighted}
    u^\pm \in \braket{x}^2 L_t^1L_x^\infty \And \partial_x u^\pm \in \braket{x} L_t^1 L_x^\infty.
  \end{equation}
  Define $\zeta(x) \coloneqq \braket{x}^{-4}$, noting that
  \begin{equation}
    \label{eq:weight-derivs}
    \abs{\zeta'} \lesssim \braket{x}^{-1} \zeta \And \abs{\zeta''} \lesssim \braket{x}^{-2} \zeta.
  \end{equation}
  Also, let $(F_\eps)_{\eps > 0}$ be a smooth, convex $L^\infty$-approximation of $F_0(s) \coloneqq s_+ = \max\{s, 0\}$.
  We construct this family so that $F_\eps = F_0$ on $(-\eps, \eps)^\cc$ and
  \begin{equation}
    \label{eq:near-zero}
    |s F_\eps'(s)| \lesssim F_\eps(s)\qquad\text{for all $s$}.
  \end{equation}
  
  Now let $v \coloneqq \tfrac{1}{2}(u^- - u^+)$ and $w \coloneqq \tfrac{1}{2}(u^+ + u^-)$.
  We define
  \begin{equation*}
    I_\eps(t) \coloneqq \int_{\R} \zeta(x) (F_\eps \circ v)\left(t, x\right) \d x,
  \end{equation*}
  which is finite for almost all $t$ because $v \in \braket{x}^2 L_t^1 L_x^\infty$.
  Because $\partial_t u^\pm \in \braket{x}^2 L_t^1L_x^\infty$, $I_\eps$ is differentiable almost everywhere and
  \begin{equation*}
    \dot{I}_\eps = \int_{\R} \zeta (F_\eps' \circ v) \partial_t v.
  \end{equation*}
  Deploying \eqref{eq:super-sub} and writing $u^\pm$ in terms of $w$ and $v$, we find
  \begin{equation}
    \label{eq:time-deriv}
    \dot{I}_\eps \leq \int_{\R} \zeta (F_\eps' \circ v) \left[w\partial_x^2 v + v\partial_x^2 w \right].
  \end{equation}
  The second term in \eqref{eq:time-deriv} is easy to manage: using \eqref{eq:near-zero}, we find
  \begin{equation}
    \label{eq:first}
    \int_{\R} \zeta (F_\eps' \circ v) v\partial_x^2 w  \lesssim I_\eps\sup_{x \in \R} |\partial_x^2w|.
  \end{equation}
  For the first term in \eqref{eq:time-deriv}, we integrate by parts, writing
  \begin{align*}
    \int_{\R} \zeta (F_\eps' \circ v) w\partial_x^2 v &= - \int_{\R} \partial_x[\zeta w (F_\eps'\circ v)] \partial_x v\\
                                                      &= - \int_{\R} \big[\partial_x(\zeta w) \partial_x(F_\eps \circ v) + \zeta w(F_\eps''\circ v) (\partial_x v)^2\big]\\
                                                      &\leq \int_{\R} (F_\eps \circ v) \partial_x^2(\zeta w).
  \end{align*}
  In the last step we used the convexity of $F_\eps$ and the positivity of $\zeta w$.
  Using this and \eqref{eq:weight-derivs}, we find
  \begin{equation*}
    \int_{\R} \zeta (F_\eps' \circ v) w\partial_x^2 v \lesssim I_\eps \sup_{x \in \R} \left[\braket{x}^{-2} w + \braket{x}^{-1} \abs{\partial_xw} + |\partial_x^2 w|\right].
  \end{equation*}
  Using this and \eqref{eq:first} in \eqref{eq:time-deriv}, we obtain
  \begin{equation*}
    \dot{I}_\eps \lesssim I_\eps \sup_{x \in \R} \left[\braket{x}^{-2} w + \braket{x}^{-1} \abs{\partial_xw} + |\partial_x^2 w|\right].
  \end{equation*}
  By \eqref{eq:super-sub-weighted}, the supremum is integrable in time.
  For all $t \in [0, T],$ Gr\"onwall's inequality yields
  \begin{equation*}
    I_\eps(t) \leq \exp\left(C \int_0^t \sup_{x \in \R} \left[\braket{x}^{-2} w + \braket{x}^{-1} \abs{\partial_xw} + |\partial_x^2 w|\right]\right) I_\eps(0)
  \end{equation*}
  for some $C \in \R_+$.
  We now take $\eps \searrow 0$ and observe that $I_\eps(0) \to 0$ because $u^-(0,\anon) \leq u^+(0,\anon)$ almost everywhere.
  It follows that
  \begin{equation*}
    \limsup_{\eps \to 0^+} I_\eps(t) = 0 \ForAll t \in [0, T].
  \end{equation*}
  Since $F_\eps(s) \to s_+$ pointwise as $\eps \searrow 0$, we have $v_+ \equiv 0$.
  That is, $u^- \leq u^+$.
\end{proof}
Of course, a comparison principle yields uniqueness.
\begin{proof}[Proof of Theorem~\textup{\ref{thm:unique}}]
  Let $u^1$ and $u^2$ be two strong solutions of \eqref{eq:main} on a time interval $[0, T)$ with the same initial data.
  Then $u^- = u^1$ and $u^+ = u^2$ satisfy the hypotheses of Proposition~\ref{prop:comparison}, so $u^1 \leq u^2$.
  By symmetry, $u^1 = u^2$.
\end{proof}

\subsection{Existence and regularity}
We next use the comparison principle to show local-in-time strong existence.
As our equation is both degenerate and quasilinear and our solutions are unbounded, we take some care with the argument.
\begin{proposition}
  \label{prop:local}
  If $u_0 \in \m{C}_\loc^1(\R)$ satisfies $\partial_x^2 u_0 \in L^\infty(\R)$, then \eqref{eq:main} admits a strong solution $u$ on the time interval $\big[0, 2(\esssup \partial_x^2 u_0)^{-1}\big)$ and $\partial_x^2 u \in L^\infty([0, T] \times \R)$  for all $T < 2(\esssup \partial_x^2 u_0)^{-1}$.
\end{proposition}
\begin{proof}
  Let $A \coloneqq \tfrac{1}{2} \esssup |\partial_x^2 u_0|, a \coloneqq \tfrac{1}{2} \esssup \partial_x^2 u_0$, $b \coloneqq \partial_x u_0(0)$, and $c \coloneqq u_0(0)$.
  Integrating twice in space, we have
  \begin{equation*}
    u_0 \leq a x^2 + bx + c.
  \end{equation*}
  Note that
  \begin{equation*}
    Q(t, x) \coloneqq (1 - at)^{-1} (ax^2 + bx + c)
  \end{equation*}
  is a strong solution of \eqref{eq:main} on the time interval $[0, a^{-1})$.

  Given $M > 0$, we define a bounded modification $u_0^M$ of $u_0$ satisfying the following properties:
  \begin{equation}
    \label{eq:mod}
    u_0^M|_{[-M,M]} = u_0|_{[-M, M]}, \quad u_0^M \geq 0, \And -2A \leq \partial_x^2 u_0^M \leq 2a.
  \end{equation}
  We demonstrate how to do so on $[M, \infty)$.
  The method extends to $(-\infty, M]$ by symmetry.

  If $\partial_x u_0(M) = 0$, we extend $u_0$ by the constant value $u_0(M)$ to $[M, \infty)$.
  Next suppose $\partial_x u_0(M) < 0$.
  We extend $u_0$ in a $\m{C}^1$ fashion by a quadratic $q_+$ with $q_+'' = 2a$ until the point $z_+ > M$ at which $q_+'(z_+) = 0$.
  Thereafter, we extend $u_0^M$ by the constant value $q_+(z_+)$ to the remainder of the ray $[M, \infty)$.
  So we have
  \begin{equation}
    \label{eq:mod-details}
    u_0^M =
    \begin{cases}
      u_0 & \text{on } [-M, M],\\
      q_+ & \text{on } [M, z_+],\\
      q_+(z_+) & \text{on} [z_+, \infty).
    \end{cases}
  \end{equation}
  The one-sided bound on $\partial_x^2 u_0$ ensures that $u_0 \leq q_+.$
  Because $u_0 \geq 0$, we have $q_+ \geq 0$, which confirms \eqref{eq:mod} in this case.
  
  Finally, suppose $\partial_x u_0(M) > 0$.
  Then we extend $u_0$ in a $\m{C}^1$ fashion by a quadratic $q_-$ with $q_-'' = -2A$ until the point $z_- > M$ at which $q_-'(z_-) = 0$.
  Thereafter, we extend $u_0^M$ by the constant value $q_-(z_-)$ to the remainder of the ray $[M, \infty)$.
  So $u_0^M$ satisfies \eqref{eq:mod-details} with $(q_-, z_-)$ in place of $(q_+, z_+)$.
  Now, the nonnegativity of $u_0^M$ is not in question on $[M, \infty)$, so \eqref{eq:mod} holds on $[M, \infty)$.
  Applying a symmetric procedure on $(-\infty, -M]$, we obtain the desired modification $u_0^M$ of $u_0$.

  Given $\eps > 0$, let $u_0^{\eps, M} \coloneqq u_0^M + \eps$.
  Finally, given $\delta > 0$, let $u_0^{\eps,M,\delta}$ denote a mollification of $u_0^{\eps, M}$ to scale $\delta$ by a suitable smooth approximation of the identity.
  Then $u_0^{\eps,M,\delta}$ is bounded, uniformly positive, and smooth.
  By Theorem~V.8.1 of~\cite{LSU68}, \eqref{eq:main} admits a smooth, bounded, global-in-time classical solution $u^{\eps, M, \delta}$ with initial condition $u_0^{\eps, M, \delta}$.
  To apply the theorem, we use the divergence-form formulation \eqref{eq:div-form}.
  We also modify the diffusivity outside $[\inf u_0^{\eps,M,\delta}, \sup u_0^{\eps,M,\delta}]$ to be uniformly bounded and positive; \emph{a posteriori}, the comparison principle implies that our solution remains between its initial infimum and supremum, and thus also solves~\eqref{eq:div-form}.
  
  Differentiating \eqref{eq:main} twice in space and letting $S \coloneqq \partial_x^2 u$, we arrive at the equation
  \begin{equation}
    \label{eq:second}
    \partial_t S = \frac{1}{2} u \partial_x^2 S + (\partial_x u) \partial_x S + \frac{1}{2}S^2
  \end{equation}
  for the second derivative.
  Now, $-2A$ and $\bar{S}(t) \coloneqq 2a(1 - at)^{-1}$ are sub- and supersolutions of \eqref{eq:second}, respectively.
  Moreover, by \eqref{eq:mod}, $-2A \leq \partial_x^2 u_0^{\eps,M,\delta} \leq \bar{S}(0)$.
  When $u = u^{\eps,M,\delta}$, the coefficients of \eqref{eq:second} are uniformly smooth, as is $S = \partial_x^2 u^{\eps,M,\delta}$.
  It thus follows from the classical comparison principle (for example, Corollary~2.5 of~\cite{L96}) that
  \begin{equation}
    \label{eq:second-bd}
    -2A \leq \partial_x^2 u^{\eps,M,\delta} \leq \bar{S}.
  \end{equation}
  We now take the mollification scale $\delta \searrow 0$.
  Intermediate Schauder estimates such as Theorem~4.29 of \cite{L96} ensure that we can extract a subsequential limit $u^{\eps, M}$ solving \eqref{eq:main} classically with data $u_0^{\eps, M}$
  Moreover, the inequalities \eqref{eq:second-bd} pass through the limit:
  \begin{equation}
    \label{eq:second-bd-limit}
    -2A \leq \partial_x^2 u^{\eps,M} \leq \bar{S}.
  \end{equation}
  Hence $u^{\eps, M}$ is a strong solution of \eqref{eq:main} on $[0, a^{-1})$; by Theorem~\ref{thm:unique}, it is unique.
  (In fact, one can combine Krylov--Safonov and Schauder estimates to show that $u^{\eps,M}$ is the unique uniformly positive bounded weak solution of \eqref{eq:main} with data $u_0^{\eps, M}$.)
  
  Now, Proposition~\ref{prop:comparison} implies that
  \begin{equation}
    \label{eq:barriers}
    \eps \leq u^{\eps, M} \leq Q.
  \end{equation}
  That is, $u^{\eps,M}$ is locally bounded independent of $\eps$ and $M$.
  Taking $M \to \infty$, the local intermediate Schauder estimate \cite[Theorem~4.29]{L96} again provides a subsequential limit $u^\eps$ solving \eqref{eq:main} classically with data $u_0 + \eps$.
  Because the bound \eqref{eq:second-bd-limit} is independent of $M$, it extends to the limit, so $u^\eps$ is again a strong solution on the time interval $[0, a^{-1})$.
  Thus $u^\eps$ is unique.
  By \eqref{eq:second-bd-limit} and \eqref{eq:barriers}, it satisfies
  \begin{equation}
    \label{eq:unbounded-bds}
    0 \leq u^\eps \leq Q \And -2A \leq \partial_x^2 u^\eps \leq \bar{S}.
  \end{equation}
  Now, the initial conditions $(u_0 + \eps)_{\eps > 0}$ are increasing in $\eps$ and the solutions $(u^\eps)_{\eps > 0}$ are strong.
  Hence by Proposition~\ref{prop:comparison}, the family $(u^\eps)_\eps$ is also increasing in $\eps$.
  We can thus consider the pointwise limit $u \coloneqq \lim_{\eps \searrow 0} u^\eps$.
  Once more, the bounds \eqref{eq:unbounded-bds} pass through the limit and yield $0 \leq u \leq Q$ and $-2A \leq \partial_x^2 u \leq \bar{S}$ a.e.
  In~\cite{DPL87}, the authors show that $u$ is a weak solution of \eqref{eq:main}.
  It follows that $u$ is in fact a strong solution on the interval $[0, a^{-1})$ and $\partial_x^2 u \in L^\infty([0, T] \times \R)$ for all $T \in (0, a^{-1})$.
\end{proof}
\begin{corollary}
  \label{cor:regular}
  If $u$ is a strong solution of \eqref{eq:main} on a time interval $[0, T)$, then $u$ is differentiable in $x$ on $(0, T) \times \R$.
  Moreover,
  \begin{equation*}
    \sup_{t \in [T_1,T_2]} \op{Lip}[\partial_xu(t, \anon)] < \infty \quad \text{ or equivalently } \quad \partial_x^2 u \in L^\infty([T_1, T_2] \times \R)
  \end{equation*}
  for all $0 < T_1 \leq T_2 < T$.
  If $\partial_x^2 u_0 \in L^\infty(\R)$, then the same holds with $T_1 = 0$.
\end{corollary}
\noindent
So although $\partial_x^2u$ is \emph{a priori} merely $L^1$ in time, it is in fact uniformly bounded on compact subsets of $(0,T)$---blow-up can only happen near the times $0$ or $T$.
\begin{proof}
  Let $A(t) \coloneqq \tfrac{1}{2} \esssup \partial_x^2u (t, \anon)$, which is defined and finite for a.e. $t \in (0, T)$.
  Take $s \in (0, T)$ such that $A(s) < \infty$.
  In the proof of Proposition~\ref{prop:local}, we showed that $A$ is defined and finite on the time interval $\m{I}(s) \coloneqq [s, \min\{s + A(s)^{-1}, T\})$.
  More precisely,
  \begin{equation}
    \label{eq:second-upper}
    A(t) \leq \frac{1}{A(s)^{-1} - (t - s)} \ForAll t \in \m{I}(s).
  \end{equation}
  For the sake of contradiction, suppose there exists $T_* \in (0, T)$ such that
  \begin{equation*}
    \adjustlimits \lim_{\eps \searrow 0} \esssup_{[T_* - \eps, T_* + \eps]} A = \infty.
  \end{equation*}
  Then \eqref{eq:second-upper} implies that $A(s) \geq (T_* - s)^{-1}$ for a.e. $s \in [0, T_*)$.
  This is not integrable in time, which contradicts the definition of strong solutions.
  We conclude that
  \begin{equation*}
    \esssup_{[T_1, T_2] \times \R} |\partial_x^2 u| < \infty
  \end{equation*}
  for all $0 < T_1 \leq T_2 < T$.
  Then \eqref{eq:second-upper} shows that $A$ is actually well-defined on the entire interval $(0, T)$.
  Integrating $\partial_x^2 u$ in $x$, we see that $u$ is differentiable in $x$ and
  \begin{equation*}
    \sup_{t \in [T_1,T_2]} \op{Lip}[\partial_xu(t, \anon)] = 2\sup_{[T_1, T_2]} A < \infty.
  \end{equation*}
  Finally, when $\partial_x^2 u_0 \in L^\infty(\R)$, the remaining claim of the corollary follows from Proposition~\ref{prop:local}.
\end{proof}
Thus far, we have focused on the boundedness of $\partial_x^2 u$.
It turns out this implies the root-Lipschitz property.
\begin{lemma}
  \label{lem:strong-root-Lip}
  If $w \in \m{C}^1_{\loc}(\R)$ satisfies $w'' \in L^\infty(\R)$ and $w \geq 0$, then
  \begin{equation*}
    \op{Lip} \sqrt{w} \le \sqrt{\sup w''}.
  \end{equation*}
\end{lemma}
\begin{proof}
  Consider $x \in \R$ such that $w(x) > 0$ and note that
  \begin{equation}
    \label{eq:sqrt-deriv-point}
    (\sqrt{w})'(x) = \frac{w'(x)}{2 \sqrt{w(x)}}.
  \end{equation}
  We may assume without loss of generality that $w'(x) \geq 0$ (otherwise, reverse space).
  Let $A \coloneqq \sup w''$.
  Then for all $y\in [x-w'(x)/(2A), x]$, we have $w'(y) \ge w'(x)/2$.
  Since $w \geq 0$, this implies that
  \begin{equation*}
    w(x) \ge \int_{x-w'(x)/(2A)}^x w'(y) \d y \ge \frac{w'(x)^2}{4A}.
  \end{equation*}
  Rearranging and taking the square root, \eqref{eq:sqrt-deriv-point} yields
  \begin{equation*}
    |(\sqrt{w})'(x)| \leq \sqrt{A},
  \end{equation*}
  for all $x$ such that $w(x)>0$. This implies the statement of the lemma.
\end{proof}
\begin{corollary}
  \label{cor:strong-implies-RL}
  If $u$ is a strong solution of \eqref{eq:main} on a time interval $[0, T)$, then for all $0 < T_1 \leq T_2 < T$, we have
  \begin{equation}
    \label{eq:local-root-Lip}
    \sup_{t \in [T_1, T_2]} \op{Lip}\sqrt{u(t, \anon)} < \infty.
  \end{equation}
  If, moreover, $\partial_x^2 u_0 \in L^\infty(\R)$, then \eqref{eq:local-root-Lip} holds with $T_1 = 0$.
\end{corollary}
\begin{proof}
  This follows from Corollary~\ref{cor:regular} and Lemma~\ref{lem:strong-root-Lip}.
\end{proof}
As these results indicate, there is a tight relationship between strong solutions and the root-Lipschitz property.
This will become even more apparent in Section~\ref{sec:existence}, wherein we derive these properties simultaneously from other bounds.
This link leads us to speculate that the root-Lipschitz condition may be \emph{equivalent} to our notion of strong solutions.
We are led to ask: are root-Lipschitz solutions of \eqref{eq:main} unique?
Precisely, if $\op{Lip}\sqrt{u_0} < \infty$, does \eqref{eq:main} admit a unique weak solution $u$ on a nonempty time interval $[0, \bar{T})$ such that $u$ satisfies \eqref{eq:root-Lip} for all $T \in [0, \bar{T})$?
We leave this to future investigation.

\subsection{The zero set}
We now consider the structure of the zero set where \eqref{eq:main} degenerates.
\begin{proposition}
  \label{prop:zeros}
  If $u$ is a strong solution of \eqref{eq:main} on $[0, T)$, then for all $t \in [0, T)$,
  \begin{equation*}
    \{u(t, \anon) = 0\} = \{u_0 = 0\}.
  \end{equation*}
\end{proposition}
\begin{proof}
  We first show that $u$ does not develop new zeros.
  Suppose $\partial_x^2 u_0 \in L^\infty(\R)$.
  Then $\exp(-\|\partial_x^2 u_0\|_{\infty} t) u_0$ is a subsolution of \eqref{eq:main} in the sense of Proposition~\ref{prop:comparison}.
  Thus our comparison principle implies that $u \geq \exp(-\sup |\partial_x^2 u_0| t) u_0$.
  In particular, $\{u_0 > 0\} \subset \{u(t, \anon) > 0\}$.
  Now consider a general strong solution, which need not have $\partial_x^2 u_0 \in L^\infty(\R)$.
  By Corollary~\ref{cor:regular}, $\partial_x^2u(t, \anon) \in L^\infty(\R)$ for all $t \in (0, T)$.
  Thus by the reasoning above, $\{u(s, \anon) > 0\} \subset \{u(t, \anon) > 0\}$ for all $0 < s \leq t < T$.
  Now, $u$ is continuous, so if $u_0(x_0) > 0$, we must have $u(t, x_0) > 0$ for sufficiently small $t$.
  It follows that $\{u_0 > 0\} \subset \{u(t, \anon) > 0\}$ and hence $\{u_0 = 0\} \supset \{u(t, \anon) = 0\}$ for all $t \in (0, T)$.

  We next show that $u$ does not lose zeros.
  Suppose $u_0(x_0) = 0$.
  Let $\theta \in \m{C}_\cc^\infty(\R)$ satisfy $\int \theta = 1$ and define $\theta_\eps(x) \coloneqq \eps^{-1}\theta\big(\tfrac{x - x_0}{\eps}\big)$ for $\eps > 0$.
  Multiplying \eqref{eq:main} by $\theta_\eps$ and integrating in time and space, we find
  \begin{equation*}
    \int_{\R} u(t, x) \theta_\eps(x) \d x = \int_{\R} u_0(x) \theta_\eps(x) \d x + \frac{1}{2} \int_{[0, t] \times \R} \partial_x^2 u(s, x) u(s, x) \theta_\eps(x) \d x \ds s
  \end{equation*}
  for all $t \in (0, T)$.
  Hence
  \begin{equation*}
    \int_{\R} u(t, \anon) \theta_\eps \leq \int_{\R} u_0 \theta_\eps + \frac{1}{2} \int_0^t \|\partial_x^2 u(s, \anon)\|_{L^\infty(\R)} \int_{\R} u(s, x) \theta_\eps(x) \d x \ds s.
  \end{equation*}
  Applying Gr\"onwall to the quantity $\int_{\R} u(t, \anon) \theta_\eps$, we find
  \begin{equation*}
    \int_{\R} u(t, \anon) \theta_\eps \leq \exp\left(\frac{1}{2} \int_0^t \|\partial_x^2 u(s, \anon)\|_{L^\infty(\R)} \d s\right) \int_{\R} u_0 \theta_\eps.
  \end{equation*}
  Because $\partial_x^2 u \in L_t^1 L_x^\infty$, the time integral is finite and independent of $\eps$.
  Taking $\eps \to 0$ and using the continuity of $u$, we see that $u(t, x_0) = 0$, as desired.
  This shows that $\{u_0 = 0\} \subset \{u(t, \anon) = 0\}$ and concludes the proof.
\end{proof}
We use the strong solution property to prove both directions of containment in Proposition~\ref{prop:zeros}.
As noted in the introduction,~\cite{DPL87} constructs weak solutions of \eqref{eq:main} with growing zero sets, while~\cite{Ugh86} constructs a solution in which an isolated zero of $u_0$ disappears.
Therefore, Proposition~\ref{prop:zeros} fails for weak solutions.

In fact, we anticipate that zero preservation is sufficient for uniqueness.
That is, given Lipschitz initial data $u_0$, we expect that there is a unique weak solution that preserves the zeros of $u_0$.
By Proposition~\ref{prop:zeros}, this solution would coincide with the unique strong solution so long as the latter exists.
However, it is not clear to us how to exploit zero preservation alone in a proof of uniqueness.
For this reason, we choose to work within the strong framework, and leave the above points to future investigation.

We next show that the evolution \eqref{eq:main} occurs independently on intervals separated by zeros.
This is a form of locality.
\begin{proposition}
  \label{prop:weak-cut-paste}
  Suppose $u$ is a weak solution of \eqref{eq:main} satisfying $u|_{\partial I} = 0$ for some nonempty open interval $I \subset \R$.
  Then $\tbf{1}_I u$ is a weak solution of \eqref{eq:main}.
  Conversely, let $\s{I}$ be a countable collection of disjoint nonempty open intervals $I \subset \R$.
  Suppose $(u_I)_{I \in \s{I}}$ is a collection of weak solutions of \eqref{eq:main} such that $u_I|_{I^c} = 0$.
  If the sum $u \coloneqq \sum_{I \in \s{I}} u_I$ satisfies $u \in \m{C}_t W_{\loc, x}^{1,\infty}$, then $u$ is a weak solution of \eqref{eq:main}.
\end{proposition}
\noindent
Informally, one can ``cut and paste'' weak solutions at locations where they vanish.
\begin{proof}
  This is essentially Theorem~2 of~\cite{DPL87}.
  Indeed, the proof of that theorem shows that if $u$ is a weak solution of \eqref{eq:main} such that $u|_{\partial I} = 0$, then \eqref{eq:weak} holds when the test function $\psi$ is replaced by $\tbf{1}_I \psi$.
  This implies the ``cut'' portion of the proposition: $\tbf{1}_I u$ is also a weak solution of \eqref{eq:main}.
  The ``paste'' portion concerning the sum $u = \sum_I u_I$ follows from the decomposition $\tbf{1}_{\bigcup I} \psi = \sum_I \tbf{1}_I \psi$.
\end{proof}
In turn, we obtain a strong formulation.
\begin{corollary}
  \label{cor:strong-cut-paste}
  Suppose $u$ is a strong solution of \eqref{eq:main} satisfying $u_0|_{\partial I} = 0$ for some nonempty open interval $I \subset \R$.
  Then $\tbf{1}_I u$ is a strong solution of \eqref{eq:main}.
  Conversely, let $\s{I}$ be as above.
  Suppose $(u_I)_{I \in \s{I}}$ is a collection of strong solutions of \eqref{eq:main} such that $u_I(0, \anon)|_{I^c} = 0$.
  If the sum $u \coloneqq \sum_{I \in \s{I}} u_I$ satisfies $\partial_x^2 u \in L_t^1L_x^\infty$, $\partial_x u \in \m{C}_t L_{\loc,x}^\infty$, and $u \in \braket{x}^2L_{t,x}^\infty$, then $u$ is a strong solution of \eqref{eq:main}.
\end{corollary}
\begin{proof}
  By Proposition~\ref{prop:zeros}, strong solutions preserve the zero set.
  For the ``cut'' part, it follows that $u|_{\partial I} = 0$.
  Hence by Proposition~\ref{prop:weak-cut-paste}, $\tbf{1}_I u$ is a weak solution of \eqref{eq:main}.
  It is a strong solution because $u$ is strong; note that setting $u$ to $0$ outside $I$ does not increase the Lipschitz constant of $\partial_x u$ because $\partial_x u = 0$ on $I$.

  For the ``paste'' part, we likewise see that $u_I|_{I^c} = 0$.
  By our hypothesis on $\partial_x u$, Proposition~\ref{prop:weak-cut-paste} implies that $u$ is a weak solution of \eqref{eq:main}.
  Then our hypotheses on $\partial_x^2 u$ and $u$ ensure that $u$ is a strong solution.
\end{proof}
We observe that strong uniqueness and the preservation of zeros makes the strong formulation easier to work with.
In particular, if we have a collection of disjointly-supported initial conditions $(u_0^I)_I$ that produce strong solutions $u_I$ satisfying suitable uniform bounds, we can construct a strong solution $u = \sum_I u_I$ from the initial condition $\sum_I u_0^I$.
We can thus focus on constructing strong solutions on individual intervals.
We take this approach in Section~\ref{sec:existence}.

Our next proposition shows that strong solutions to \eqref{eq:main} are ``almost classical solutions'' in the sense of \cite[Definition~1.13]{DG}.
This enables Itô's formula, which links strong solutions of \eqref{eq:main} with the forward-backward SDE~\eqref{eq:FBSDE}.
\begin{proposition}
  \label{prop:C1}
  Let $u$ be a strong solution of \eqref{eq:main} on $[0,T)$.
  Then $u$ is continuously differentiable on $(0,T)\times\R$ and moreover satisfies the differential equation \eqref{eq:main} almost everywhere on $(0, T) \times \R$.
\end{proposition}
\begin{proof}
  To show that $u$ is continuously differentiable, it suffices to show that $\partial_t u$ and $\partial_x u$ exist and are continuous at each $(t,x)\in (0,T)\times\R$.
  By Lemma~\ref{lem:positive}, this is true at each $(t,x)$ such that $u(t,x)>0$.
  So consider some $(t,x)\in (0,T)\times\R$ such that $u(t,x)=0$.
  By Proposition~\ref{prop:zeros}, we have $u(s,x)=0$ for all $s\in (0,T)$, so $\partial_t u(t,x)=0$.
  Moreover, Corollary~\ref{cor:regular} ensures that $\partial_x u(t,x)$ exists; since $u$ is nonnegative, we have in fact have $\partial_x u(t, x) = 0$.

  We now show continuity at $(t, x)$, beginning with $\partial_x u$.
  Let $\m{I} \Subset (0, T)$ be a neighborhood of $t$.
  Fix $(s, y) \in \m{I} \times \R$ and observe that $\partial_x u(t, x) = 0 = \partial_xu(s, x)$.
  Therefore
  \begin{equation*}
    |\partial_x u(s,y)-\partial_x u(t,x)| = |\partial_x u(s,y)-\partial_x u(s,x)|\le \sup_{{s \in \m{I}}} \op{Lip}[\partial_x u(s, \anon)] |x-y|.
  \end{equation*}
  By Corollary~\ref{cor:regular}, $\partial_xu$ is continuous at $(t, x)$.

  We now consider the continuity of $\partial_t u$ at $(t,x)$.
  Let $\eps>0$.
  Since $u$ is continuous, there exists a neighborhood $\m{N} \Subset \m{I} \times \R$ such that $u|_{\m{N}} < \eps$.
  Take $(s, y) \in \m{N}$.
  If $u(s, y) = 0$, then we argued above that $\partial_tu(s,y) = 0 = \partial_tu(t,x)$.
  On the other hand, if $u(s, y) > 0$, then Lemma~\ref{lem:positive} implies that $u$ is smooth near $(s, y)$ and satisfies \eqref{eq:main} pointwise.
  Hence
  \begin{equation*}
    |\partial_tu(s,y)-\partial_tu(t,x)|=\frac12|u(s,y)\partial_x^2u(s,y)| \le \frac \eps 2 \sup_{s \in \m{I}} \|\partial_x^2u(s, \anon)\|_{L^\infty(\R)}.
  \end{equation*}
  By Corollary~\ref{cor:regular}, $\partial_t u$ is continuous at $(t, x)$.

  Finally, we show that $u$ satisfies the differential equation \eqref{eq:main} almost everywhere on $(0,T)\times\R$. On the set $\{u>0\}$, \eqref{eq:main} holds everywhere by Lemma~\ref{lem:positive}.
  On the set $\{u=0\}$, Proposition~\ref{prop:zeros} implies that both sides of \eqref{eq:main} are zero whenever the right side exists, and the right side exists almost everywhere since $\partial_x^2u\in L^\infty_{\loc}\big((0,T)\times\R\big)$.
  This completes the proof.
\end{proof}
\begin{remark}
  In fact, the above proof shows that although $u$ may not be twice differentiable everywhere, the product $u \partial_x^2 u$ has a continuous version.
  If we interpret \eqref{eq:main} with this continuous version on the right, then \eqref{eq:main} holds pointwise \emph{everywhere} in $(0, T) \times \R$.
  On the other hand, we do wish to emphasize that $u$ need not be twice differentiable everywhere.
  For example, $u(t, x) \coloneqq \tfrac{x_+^2}{1 - t}$ is a strong solution of \eqref{eq:main} on the time interval $[0, 1)$ that is not twice differentiable at $x = 0$.
\end{remark}

\subsection{Finite-time blow-up}
We now turn to finite-time blow-up for strong solutions.
This should be quite common, but to rigorously prove it, we need good control over various quantities of interest.
This motivates the following lemma.
\begin{lemma}
  \label{lem:smooth}
  If $u$ is a strong solution of \eqref{eq:main} on an interval $[0, \bar T)$ with $u_0 \in \m{C}_\cc^\infty(\R)$, then $u$ is smooth on $\R \times [0, \bar T)$.
\end{lemma}
\begin{proof}
  By Proposition~\ref{prop:zeros}, $u$ remains compactly supported.
  Fix $T \in (0, \bar{T})$; in the following, we work on $[0, T] \times \R$.
  Because $u$ is a strong solution with initially bounded second derivative, Corollary~\ref{cor:regular} implies that $u$, $\partial_xu$, and $\partial_x^2 u$ are all bounded and compactly supported.
  In particular, they each belong to $L_t^\infty L_x^p$ for all $p \in [1, \infty]$.
  
  Given $k \geq 3$, we formally differentiate \eqref{eq:main} $k$ times in space.
  Letting $z^{(k)} \coloneqq \partial_x^k u$, we find
  \begin{equation}
    \label{eq:higher}
    \partial_t z^{(k)} = \frac{1}{2}u \partial_x^2 z^{(k)} + \frac{k}{2} (\partial_x u) \partial_x z^{(k)} + \frac{k(k-1) + 2}{4} (\partial_x^2 u) z^{(k)} + F_k,
  \end{equation}
  where $F_k$ has the form
  \begin{equation*}
    F_k = \sum_{\ell = 3}^{k-1} a_{k\ell} (\partial_x^{k + 2 - \ell} u) (\partial_x^\ell u)
  \end{equation*}
  for certain coefficients $a_{k\ell} \in \R$.

  We show by induction that $z^{(k)} \in L_t^\infty L_x^2$ for all $k \geq 3$.
  The lemma will then follow from Sobolev embedding.
  For the base case, note that $F_3 = 0$.
  Multiplying \eqref{eq:higher} by $z^{(3)}$ and integrating by parts repeatedly, we obtain the energy estimate
  \begin{equation*}
    \der{}{t} \int_{\R} \big|z^{(3)}\big|^2 \lesssim \int_{\R} |\partial_x^2 u| \big|z^{(3)}\big|^2.
  \end{equation*}
  Since $\partial_x^2 u \in L_{t,x}^\infty$, Gr\"onwall yields $z^{(3)} \in L_t^\infty L_x^2$ as desired.
  In fact, a similar calculation implies that $z^{(3)} \in L_t^\infty L_x^p$ for all $p \in [1, \infty)$.

  Now fix $k \geq 4$ and suppose we have shown $z^{(\ell)} \in L_t^\infty L_x^2$ for all $\ell < k$.
  Similar manipulations of \eqref{eq:higher} and Young's inequality yield
  \begin{equation}
    \label{eq:Young}
    \der{}{t} \int_{\R} \big|z^{(k)}\big|^2 \lesssim \int_{\R} |\partial_x^2 u| \big|z^{(k)}\big|^2 + \int_{\R} \big|F_k z^{(k)}\big| \lesssim \int_{\R} \big|z^{(k)}\big|^2 + \int_{\R} \abs{F_k}^2.
  \end{equation}
  It thus suffices to show that $F_k \in L_t^\infty L_x^2$.
  By the inductive hypothesis and Sobolev embedding, $z^{(\ell)} \in L_{t,x}^\infty$ for all $\ell < k - 1$.
  When $k \geq 5$, this implies that $F_k \in L_t^\infty L_x^2$, as desired.
  The case $k = 4$ is special, for then $F_4 = 2\abs{z^{(3)}}^2$.
  We argued above that $z^{(3)} \in L_t^\infty L_x^4$, so again we obtain $F_4 \in L_t^\infty L_x^2$.
  Having checked every case, \eqref{eq:Young} and Gr\"onwall yield $z^{(k)} \in L_t^\infty L_x^2$.
  This completes the induction and the proof.
\end{proof}
\begin{remark}
  Strong solutions likely preserve smoothness beyond the case of compact support.
  However, the resulting regularity theory would need to handle diffusion of order $\braket{x}^2$ and drift of order $\braket{x}$.
  Both of these are ``borderline'' and should barely yield well-posedness.
  (They correspond to exponential martingales and exponential trajectories, respectively.)
  This smoothness result is not the primary aim of this work, so we leave the general problem to future investigation.
\end{remark}
\begin{proposition}
  \label{prop:blow-up}
  Let $u_0 \in \m{C}_\cc^\infty(\R)$ satisfy $u_0 \geq 0$, $u_0(0) = 0,$ and $u_0''(0) > 0$.
  Then \eqref{eq:main} does not admit a strong solution past the time $2u_0''(0)^{-1}$.
\end{proposition}
\begin{proof}
  Let $T_* > 0$ denote the maximal time of existence of the strong solution, which may \emph{a priori} be infinite but is positive by Proposition~\ref{prop:local}.
  By Lemma~\ref{lem:smooth}, we can differentiate \eqref{eq:main} twice to obtain \eqref{eq:second} for $S \coloneqq \partial_x^2 u$.
  We know that $u$ preserves the zeros of $u_0$ (Proposition~\ref{prop:zeros}) so $u(t, 0) = 0$ and smoothness yields $\partial_x u(t, 0) = 0$ for all $t \in [0, T_*)$.
  Evaluating \eqref{eq:second} at $x = 0$, we find the closed Riccati equation
  \begin{equation}
    \label{eq:Riccati}
    \der{}{t} \partial_x^2u(t, 0) = \frac{1}{2}[\partial_x^2u(t, 0)]^2.
  \end{equation}
  This ODE holds on the time interval $(0, T_*)$.
  If $T_* < 2 u_0''(0)^{-1}$, we are done.
  Otherwise, we can integrate \eqref{eq:Riccati} to obtain
  \begin{equation*}
    \partial_x^2u(t, 0) = 2\left[2u_0''(0)^{-1} - t\right]^{-1} \ForAll t \in \big[0, 2u_0''(0)^{-1}\big).
  \end{equation*}
  This blow-up is not time-integrable, so $\partial_x^2 u \not \in L_t^1 L_x^\infty$ on $[0, 2u_0''(0)^{-1}) \times \R$.
  By the definition of strong solutions, $T_* \leq 2u_0''(0)^{-1}$.
\end{proof}
This proposition sets an upper bound on the lifetime of these strong solutions.
If $u_0$ satisfies Hypothesis~\ref{hyp:init} with $\kappa = \tfrac{1}{2} u_0''(0)$ (and any $\gamma$), then Theorem~\ref{thm:root-Lip} implies that this upper bound is sharp.
However, this bound need not be sharp if $u_0$ does not satisfy said hypothesis.
This is essentially a consequence of Theorem~\ref{thm:blow-up}; although Theorem~\ref{thm:blow-up} is stated on the half-line, a variant holds on the bounded interval.
Thus if $u_0$ is sufficiently large in the interior of its support, $u$ can cease to be strong before the time $2u_0''(0)^{-1}$ in Proposition~\ref{prop:blow-up}.

\section{Existence of strong, root-Lipschitz solutions}
\label{sec:existence}

We now construct strong, root-Lipschitz solutions for any initial data satisfying Hypothesis~\ref{hyp:init}.

\subsection{Classical existence}
We use the method of vanishing viscosity to construct the solution in Theorem~\ref{thm:root-Lip}.
Given $u_0$, we seek to solve \eqref{eq:main} with initial data $u_0 + \eps$ for $\eps > 0$.
Taking $\eps \searrow 0$, we then expect to recover a solution of \eqref{eq:main} with data $u_0$.

We are therefore interested in solving \eqref{eq:main} when the initial condition is uniformly positive (as is $u_0 + \eps$).
To control the solution, we wish to bracket it by an ordered pair of sub- and supersolutions.
This construction and comparison is standard when the solution is bounded, for then \eqref{eq:main} is uniformly parabolic.
However, we stray beyond off-the-shelf theory when the solution, and thus the diffusivity, is unbounded.
We therefore justify this construction method in our setting.
\begin{lemma}
  \label{lem:sandwich}
  Let $\ubar{u} \leq \bar{u}$ be classical sub- and supersolutions of \eqref{eq:main}, respectively, on a time interval $[0, T]$.
  Suppose $\inf \ubar{u} > 0$ and $u_0 \in \m{C}_\loc^1(\R)$ satisfies
  \begin{equation*}
    \ubar{u}(0, \anon) \leq u_0 \leq \bar{u}(0, \anon).
  \end{equation*}
  Then there exists a classical solution $u$ of \eqref{eq:main} on $[0, T]$ satisfying $\ubar{u} \leq u \leq \bar{u}$.
\end{lemma}
\begin{proof}
  Fix $T > 0$.
  Given $M > 0$, take some continuous function $\psi_\pm^M \colon [0, T] \to \R_+$ such that
  \begin{equation*}
    \psi_\pm^M(0) = u_0(\pm M) \And \ubar{u}(\anon, \pm M) \leq \psi_\pm^M \leq \bar{u}(\anon, \pm M).
  \end{equation*}
  We study the following initial-boundary value problem in $(0, T) \times (-M, M)$:
  \begin{equation}
    \label{eq:box}
    \begin{cases}
      \partial_t u^M = \frac{1}{2} u^M \partial_x^2 u^M,\\
      u^M(\anon, \pm M) = \psi_\pm^M,\\
      u^M(0, \anon) = u_0.
    \end{cases}
  \end{equation}
  Because $\inf \ubar{u} > 0,$ this problem is uniformly parabolic.
  Hence if we mollify the initial-boundary data, Theorem~V.6.1 of~\cite{LSU68} states that \eqref{eq:box} admits a unique smooth solution.
  Then interior Schauder estimates such as Theorems~4.8 and 4.9 of~\cite{L96} allow us to remove the mollification and obtain a classical solution of \eqref{eq:box} that is smooth in the interior.
  Because $u_0 \in C_\loc^1$, Theorem~11.3 in~\cite{L86} implies that $u^M$ is H\"older continuous away from $x = \pm M$.

  Now, the classical comparison principle on bounded domains (e.g., Corollary~2.5 of~\cite{L96}) implies that
  \begin{equation}
    \label{eq:box-bd}
    \ubar{u} \leq u^M \leq \bar{u}.
  \end{equation}
  In particular, the sequence $(u^M)_M$ is locally bounded independent of $M$.
  Therefore the parabolic Schauder estimates mentioned above allow us to extract a locally uniform limit $u$ of $(u^M)_M$ as $M \to \infty$.
  This limit $u$ is locally H\"older regular, smooth when $t > 0$, and satisfies \eqref{eq:main}.
  By \eqref{eq:box-bd}, $\ubar{u} \leq u \leq \bar{u}$.
\end{proof}
\noindent
Note that we do not claim that $u$ is unique, although this does hold in our applications; see the three-part proof of Theorem~\ref{thm:root-Lip} below.

\subsection{Change of variables}
\label{sec:CoV}
To make use of the solution $u$ constructed in Lemma~\ref{lem:sandwich}, we need to show that it is strong (under suitable hypotheses on $\ubar{u}$ and $\bar{u}$).
We accomplish this via a change of variables that makes \eqref{eq:main} uniformly parabolic.
On each interval $I$, we will define $\zeta \colon \R \to I$ and
\begin{equation}
  \label{eq:transform}
  v(t, y) \coloneqq \zeta'(y)^{-2} u\big(t, \zeta(y)\big).
\end{equation}
In a minor abuse of notation, we use $\circ$ to denote composition in the spatial coordinate, so that $u\big(t, \zeta(y)\big) = (u \circ \zeta)(t, y)$.
Then we can compute
\begin{align}
  u \circ \zeta &= (\zeta')^2 v,\label{eq:CoV0}\\
  (\partial_x u) \circ \zeta &= \zeta' \partial_y v + 2 \zeta'' v,\label{eq:CoV1}\\
  (\partial_x^2 u) \circ \zeta &= \partial_y^2 v + \frac{3\zeta''}{\zeta'} \partial_y v + \frac{2\zeta'''}{\zeta'}v.\label{eq:CoV2}
\end{align}
Differentiating \eqref{eq:transform} in time and using \eqref{eq:main} and \eqref{eq:CoV2}, we obtain
\begin{equation}
  \label{eq:transform-ev}
  \partial_t v = \frac{1}{2}v \partial_y^2 v + \frac{3\zeta''}{2 \zeta'} v\partial_y v + \frac{\zeta'''}{\zeta'} v^2.
\end{equation}
We will choose $\zeta$ depending on $I$ and $\gamma$ so that Theorem~\ref{thm:root-Lip} reduces to showing that $v \asymp 1,$ $\partial_y v \lesssim 1$, and $\partial_y^2 v \in L_t^1 L_x^\infty$.
In each case, we use the structure of $\ubar{u}$ and $\bar{u}$ in Lemma~\ref{lem:sandwich} to prove that $v \asymp 1$.
Once this is accomplished, \eqref{eq:transform-ev} is uniformly parabolic, which is a distinct advantage over \eqref{eq:main}.
This allows us to deploy parabolic estimates to control $\partial_y v$ and $\partial_y^2 v$.

\subsection{Parabolic estimates}
We state these estimates in a more abstract form.
Consider the quasilinear parabolic equation
\begin{equation}
  \label{eq:transformed}
  \partial_t v = a(v) \partial_y^2 v + b(v, y) \partial_y v + c(v, y), \quad v(0, \anon) = v_0.
\end{equation}
We work under the following assumptions on the coefficients.
\begin{customhyp}{A}
  \label{hyp:transformed}
  The functions $a \colon \R_+ \to \R_+$ and $b,c \colon \R_+ \times \R \to \R$ are smooth.
  Moreover, $b, c$ and all their derivatives are uniformly bounded on $\m{K} \times \R$ for every compact set $\m{K} \subset \R_+.$
\end{customhyp}
We show that \eqref{eq:transformed} propagates Lipschitz regularity.
\begin{proposition}
  \label{prop:transformed}
  Under Hypothesis~\textup{\ref{hyp:transformed}}, suppose $v$ is a classical solution of \eqref{eq:transformed} on $[0, T]$ such that $\op{Lip} v_0 \leq K$ and $K^{-1} \leq v \leq K$ for some $K \in [1, \infty)$.
  Then there exists $C(K, a, b, c) \in (0, \infty)$ such that
  \begin{equation}
    \label{eq:derivs-time}
    |\partial_y v| \leq C \And |\partial_y^2 v| \leq C t^{-1/2} \quad \text{on } (0, T] \times \R.
  \end{equation}
\end{proposition}
\begin{proof}
  The diffusivity $a$ is bounded away from $0$ and $\infty$ on compact subsets of $\R_+$.
  By hypothesis, $v \asymp 1$, so \eqref{eq:transformed} is uniformly parabolic with bounded coefficients.
  Because $v_0$ is Lipschitz, Krylov--Safonov estimates up to the boundary imply that $v$ is $\beta$-H\"older on $[0, T] \times \R$ for some $\beta \in (0, 1)$; see, for instance, Corollary~7.45 in~\cite{L96}.
  We now upgrade this to Lipschitz regularity.

  By a routine approximation argument, we can assume $v$ is smooth and prove suitable estimates.
  Let $w$ solve $\partial_t w = a(v, y) \partial_y^2 w$ with $w(0, \anon) = v_0$ and $z$ solve $\partial_t z = a(v, y) \partial_y^2 z + b(v, y) \partial_y v + c(v, y)$ with $z(0, \anon) = 0$.
  Then $v = w + z$, and by the maximum principle $\norm{w}_{L^\infty} \leq K$ and $\norm{z}_{L^\infty} \leq 2K$.
  Differentiating the $w$-equation, we see that $\partial_y w$ satisfies the divergence-form equation
  \begin{equation*}
    \partial_t \partial_y w = \partial_y[a(v, y) \partial_y \partial_y w].
  \end{equation*}
  It follows from the maximum principle that $\norms{\partial_y w}_{L^\infty} \leq \norms{\partial_y v_0}_{L^\infty} \leq K$.
  For $z$, a local $W^{2,p}$ estimate such as Theorem 7.30 of~\cite{L86} yields
  \begin{equation*}
    \norm{z}_{W^{2,p}(Q_{1/2}(X))} \lesssim \norms{\partial_y v}_{L^p(Q_1(X))} + 1
  \end{equation*}
  for any $X\in\R^2$ and $p \in (1, \infty)$, where $Q_r(X)$ denotes a ball with center $X$ and radius $r$ intersected with $[0, T] \times \R$.
  By Morrey's inequality, we see that for any $\alpha\in (0,1)$,
  \begin{equation*}
    \norm{z}_{\m{C}_y^{1,\al}} \leq A(\norms{\partial_y v}_{L^\infty} + 1)
  \end{equation*}
  for some $A=A(K, a, b, c, \alpha) \in (0, \infty)$.
  Interpolating between the $L^\infty$ and $\m{C}^{1,\alpha}_y$ norms, for any $\eps>0$ there exists $B = B(\eps,K)\in(0,\infty)$ such that
  \begin{equation*}
    \norm{z}_{\m{C}_y^{0,1}} \leq \eps \norm{z}_{\m{C}_y^{1,\al}} + B(\eps, K) \leq A \eps \norms{\partial_y v}_{L^\infty} + A\eps + B.
  \end{equation*}
  Choosing $\eps = 1/(2A)$, we find
  \begin{equation*}
    \norms{\partial_y z}_{L^\infty} \leq \frac{1}{2} \norms{\partial_y v}_{L^\infty} + \frac{1}{2} + B.
  \end{equation*}
  Noting that $\abss{\partial_y v} \leq \abss{\partial_y w} + \abss{\partial_y z}$ and rearranging, we see that $\norms{\partial_y v}_{L^\infty} \leq K + B + 1$.
  It follows that $r \coloneqq \partial_y v$ is bounded.
  Differentiating \eqref{eq:transformed}, $r$ solves the equation
  \begin{equation*}
    \partial_t r = \partial_y(a \partial_yr)  + b \partial_y r + [(\partial_v b)r  + \partial_v c + \partial_y b] r + \partial_y c
  \end{equation*}
  on $(0, T] \times \R$.
  Thus weighted interior Schauder estimates such as \cite[Theorem~4.8]{L96} ensure that $\abss{\partial_y^2 v}  = \abss{\partial_y r} \leq C t^{-1/2}$ on $(0, T] \times \R$.
\end{proof}
Recall that we intend to transform \eqref{eq:main} into an equation of the form \eqref{eq:transformed} via a change of variables depending on the interval $I$.
With estimates on \eqref{eq:transformed} in hand, we now treat each interval $I$ in Hypothesis~\ref{hyp:init} separately.

\subsection{The bounded interval}

We begin with the bounded interval $I = (0, L)$ for some $L > 0$.
To construct our solution, we employ the method of vanishing viscosity (as in the proof of Proposition~\ref{prop:local}).
Given $\eps > 0$, let $u^\eps$ solve \eqref{eq:main} with $u^\eps(0, \anon) = u_0 + \eps \eqqcolon u_0^\eps$.
This Cauchy problem is well-behaved because \eqref{eq:main} is then uniformly parabolic.
In particular, there is a unique bounded global-in-time weak solution and it is classical.
By the comparison principle, the sequence $(u^\eps)_{\eps > 0}$ is increasing in $\eps$.
It follows that there is a nonnegative limit $u \coloneqq \lim_{\eps \searrow 0} u^\eps$.
We show that $u$ is the desired strong solution of \eqref{eq:main}.

We first establish quadratic behavior in the vicinity of $\partial I$.
In the following, recall that $d$ denotes the distance to $I^\cc$, so in particular $d(x)=0$ whenever $x\not\in I$.
\begin{lemma}
  \label{lem:int}
  Suppose $u_0$ satisfies \textup{Hypothesis~\ref{hyp:init}}, and in particular \eqref{eq:hyp-int}.
  Then there is a unique uniformly bounded, uniformly positive solution $u^\eps$ of \eqref{eq:main} with initial data $u_0^\eps$.
  Moreover, the sequence $(u^\eps)_\eps$ is increasing in $\eps$ and the limit $u = \lim_{\eps \searrow 0} u^\eps$ satisfies
  \begin{equation}
    \label{eq:int}
    \frac{d(x)^2}{2(\kappa^{-1} + K)} \leq u(t, x) \leq \frac{d(x)^2}{\kappa^{-1} - t} \ForAll (t, x) \in [0, \kappa^{-1}) \times \R.
  \end{equation}
\end{lemma}
\begin{proof}
  Define the nonlinear heat operator
  \begin{equation}
    \label{eq:heat-op}
    \op{NH}[v] \coloneqq \partial_t v - \frac{1}{2} v \partial_x^2 v.
  \end{equation}
  Then $v \in \m{C}_t^1W_x^{2,\infty}$ is a supersolution of \eqref{eq:main} when $\op{NH}[v] \geq 0$ and a subsolution when $\op{NH}[v] \leq 0$.
  Define
  \begin{equation*}
    \bar{u}^\eps(t, x) \coloneqq \frac{x^2 + \kappa^{-1} \eps}{\kappa^{-1} - t}.
  \end{equation*}
  Then $\op{NH}[\bar{u}^\eps] = 0$, so $\bar{u}^\eps$ solves \eqref{eq:main} when $t < \kappa^{-1}$.
  By \eqref{eq:hyp-int}, $\bar{u}^\eps(0, \anon) \geq u_0^\eps$.

  For the subsolution, take $G \in W^{2,\infty}(I)$ such that $d^2 \leq G \leq 2d^2$.
  Using a piecewise quadratic function, we can arrange $\abs{G''} = 4$.
  Extend $G$ by $0$ to $\R$.
  By \eqref{eq:hyp-int}, $u_0 \geq G/(2K)$.
  We therefore define
  \begin{equation*}
    \ubar{u}^\eps(t, x) \coloneqq \frac{G(x) + \eps}{2(t + K)}.
  \end{equation*}
  Then $\op{NH}[\ubar{u}^\eps] \leq 0$ and $\ubar{u}^\eps(0, \anon) \leq u_0^\eps$.

  It now follows from Lemma~\ref{lem:sandwich} that there exists a solution $u^\eps$ of \eqref{eq:main} with initial data $u_0^\eps$ satisfying
  \begin{equation}
    \label{eq:sandwich}
    \ubar{u}^\eps \leq u^\eps \leq \bar{u}^\eps.
  \end{equation}
  In this uniformly parabolic setting, Schauder estimates imply that $u^\eps$ is a strong solution and thus unique by Theorem~\ref{thm:unique}.
  Moreover, the comparison principle (Proposition~\ref{prop:comparison}) ensures that the sequence $(u^\eps)_\eps$ is increasing in $\eps$, so the limit $u = \lim_{\eps \searrow 0} u^\eps \geq 0$ exists.
  Taking $\eps \searrow 0$ in \eqref{eq:sandwich}, we see that
  \begin{equation*}
    \frac{d(x)^2}{2(t + K)} \leq u(t, x) \leq \frac{x^2}{\kappa^{-1} - t} \quad \text{on } [0, \kappa^{-1}) \times I.
  \end{equation*}
  Using symmetry under the map $x \mapsto L - x$, we obtain \eqref{eq:int}.
\end{proof}
This control on $u$ suffices to prove our main result on the interval.
\begin{proof}[Proof of Theorem~\textup{\ref{thm:root-Lip}} on the interval]
  Let $(u^\eps)_\eps$ be the sequence of solutions from Lemma~\ref{lem:int}.
  In \cite{DPL87}, the authors show that $u = \lim_{\eps \searrow 0} u^\eps$ is a weak solution of \eqref{eq:main}.
  We show that it is strong and root-Lipschitz on the time interval $[0, \kappa^{-1})$.

  Define $\zeta \colon \R \to I$ by $\zeta(y) \coloneqq \tfrac{L}{2}(\tanh y+1)$ and consider $v$ given by \eqref{eq:transform}, so $v = (\zeta')^{-2} u \circ \zeta$.
  This satisfies \eqref{eq:transform-ev}, which we recall for the reader's convenience:
  \begin{equation}
    \label{eq:transform-ev-int}
    \partial_t v = \frac{1}{2} v \partial_y^2 v + \frac{3\zeta''}{2\zeta'} \partial_y v + \frac{\zeta'''}{\zeta'} v^2.
  \end{equation}
  For this choice of $\zeta$, the coefficients $\zeta''/\zeta'$ and $\zeta'''/\zeta'$ in \eqref{eq:transform-ev-int} are uniformly smooth on $\R$.
  It follows that the equation \eqref{eq:transform-ev-int} satisfies Hypothesis~\ref{hyp:transformed}.
  Now, one can check that the distance $d$ to $I^\cc$ satisfies
  \begin{equation*}
    d \circ \zeta \asymp \zeta'.
  \end{equation*}
  Thus the definition \eqref{eq:transform} of $v$ implies that
  \begin{equation}
    \label{eq:int-equiv0}
    u \asymp d^2 \iff v \asymp 1.
  \end{equation}
  Assuming these equivalent conditions hold, \eqref{eq:CoV1} yields
  \begin{equation}
    \label{eq:int-equiv1}
    |\partial_x u| \lesssim d \iff |\partial_y v| \lesssim 1.
  \end{equation}
  Again assuming this and \eqref{eq:int-equiv0}, \eqref{eq:CoV2} implies that
  \begin{equation}
    \label{eq:int-equiv2}
    |\partial_x^2 u| \lesssim 1 \iff |\partial_y^2 v| \lesssim 1.
  \end{equation}
  Through these relations, we transform our investigation of $u$ into one of $v$.

  Fix $T \in (0, \kappa^{-1})$.
  By \eqref{eq:int-equiv0}, Lemma~\ref{lem:int} implies that $v \asymp 1$ on $[0, T] \times \R$.
  Now recall that Hypothesis~\ref{hyp:init} includes the assumption  $\op{Lip}\sqrt{u_0} < \infty$.
  Using the identity
  \begin{equation}
    \label{eq:sqrt-deriv}
    \partial_x \sqrt{u} = \frac{\partial_x u}{2 \sqrt{u}}
  \end{equation}
  and \eqref{eq:hyp-int}, we see that $|\partial_x u_0| \lesssim \sqrt{u_0} \lesssim d$.
  Thus \eqref{eq:int-equiv1} yields $|\partial_y v_0| \lesssim 1$, and $v$ satisfies the hypotheses of Proposition~\ref{prop:transformed}.
  It follows that $|\partial_y v| \lesssim 1$ and $|\partial_y^2 v| \lesssim t^{-1/2}$ on $(0, T] \times \R$.
  Using \eqref{eq:int-equiv1} and \eqref{eq:int-equiv2}, these estimates become
  \begin{equation}
    \label{eq:int-deriv-bd}
    |\partial_x u| \lesssim d \And |\partial_x^2 u| \lesssim t^{-1/2} \quad \text{on } (0, T] \times I.
  \end{equation}
  In particular, because $t^{-1/2}$ is integrable in time, $\partial_x^2 u \in L_t^1 L_x^\infty$ and $u$ is a strong solution of \eqref{eq:main} on $[0, \kappa^{-1})$.
  
  It only remains to check that $u$ is root-Lipschitz.
  By \eqref{eq:int}, $u \gtrsim d^2$ on $[0, T] \times I$.
  On the other hand, \eqref{eq:int-deriv-bd} states that $|\partial_x u| \lesssim d$ there.
  Hence \eqref{eq:sqrt-deriv} implies that $\partial_x \sqrt{u}$ is uniformly bounded on $[0, T] \times I$.
\end{proof}

\subsection{The whole line}

Next, we turn to solutions on the whole line.
In the previous subsection, we needed to carefully control the behavior of $u$ near $\partial I$.
There is no such boundary on the whole line, but instead we must worry about the behavior of $u$ at infinity.

A simple calculation shows that $x^2/(T - t)$ is a strong solution of \eqref{eq:main} that blows up at time $T > 0$.
This blow-up is due to large diffusion moving mass in from spatial infinity in finite time.
The situation would be even worse if $u_0$ grew faster than $x^2$, so we always assume that $u_0 \lesssim \braket{x}^2$.
In the other direction, for simplicity we assume that $u_0$ is uniformly positive.
Then \eqref{eq:main} is uniformly parabolic from below and the regularization $u^\eps$ is unnecessary.

Within the class of quadratically-bounded uniformly positive functions, we focus on initial conditions with fixed power-law growth at infinity.
That is, $u_0$ is comparable to $\braket{x}^\gamma$ for some $\gamma \in [0, 2]$.
One could also consider data with distinct growth exponents at $\pm \infty$; this would require only minor modifications to our arguments.

To begin, we show that there exists a solution with the expected growth.
\begin{lemma}
  \label{lem:line}
  Suppose $u_0$ satisfies Hypothesis~\textup{\ref{hyp:init}}, and in particular \eqref{eq:hyp-line}.
  If $\gamma = 2$, there exists a classical solution $u$ of \eqref{eq:main} on the time interval $[0, \kappa^{-1})$ satisfying
  \begin{equation}
    \label{eq:line-quad}
    K^{-1} \braket{x}^2 \leq u(t, x) \leq \frac{x^2 + \kappa^{-1}K}{\kappa^{-1} - t}.
  \end{equation}
  If $\gamma \in [0, 2)$, then there exist $D(\gamma, K) > 0$ and a classical solution $u$ of \eqref{eq:main} on the time interval $[0, \infty)$ such that
  \begin{equation}
    \label{eq:line-subquad}
    \frac{\braket{x}^\gamma}{t + K} \leq u(t, x) \leq 2K \braket{x}^\gamma + D \e^{3Kt}.
  \end{equation}
\end{lemma}
\begin{proof}
  First suppose $\gamma = 2$.
  Define
  \begin{equation*}
    \ubar{u}_2(t, x) \coloneqq K^{-1} \braket{x}^2 \And \bar u_2(t, x) \coloneqq \frac{x^2 + \kappa^{-1}K}{\kappa^{-1} - t}
  \end{equation*}
  on $[0, \kappa^{-1}) \times \R$.
  These are sub- and supersolutions of \eqref{eq:main}, respectively.
  By \eqref{eq:hyp-line}, $\ubar{u}_2(0, \anon) \leq u_0 \leq \bar{u}_2(0, \anon)$.
  Thus Lemma~\ref{lem:sandwich} yields a classical solution $u$ of \eqref{eq:main} on the time interval $[0, \kappa^{-1})$ such that $\ubar{u}_2 \leq u \leq \bar{u}_2$.
  This proves \eqref{eq:line-quad}.
  
  Now suppose $\gamma \in [0, 2)$.
  We make use of the following identity:
  \begin{equation}
    \label{eq:Japanese-deriv}
    \partial_x^2 \braket{x}^\gamma = \gamma[(\gamma - 1) x^2 + 1] \braket{x}^{\gamma - 4}.
  \end{equation}
  For the lower bound, we observe that $\partial_x^2 \braket{x}^\gamma \geq -1$.
  It follows that
  \begin{equation*}
    \ubar{u}_\gamma(t, x) \coloneqq \frac{\braket{x}^\gamma}{t + K}
  \end{equation*}
  is a subsolution of \eqref{eq:main}.
  For the supersolution, we define
  \begin{equation*}
    \bar{u}_\gamma(t, x) \coloneqq K a(t) \braket{x}^\gamma + b(t),
  \end{equation*}
  where
  \begin{equation*}
    a(t) \coloneqq 2 - \frac{1}{t + 1}
  \end{equation*}
  and $b \colon \R_+ \to \R_+$ remains to be determined.
  Note that $1 \leq a < 2$.
  Recalling the operator $\op{NH}{}$ from \eqref{eq:heat-op}, we use \eqref{eq:Japanese-deriv} to compute
  \begin{equation}
    \label{eq:gamma-super}
    \op{NH}[\bar{u}_\gamma] \geq K \dot a \braket{x}^\gamma - 4 K^2 \braket{x}^{2\gamma-2} + \dot b - 2Kb.
  \end{equation}
  Now fix $\Lambda \geq 1$.
  We wish to bound $f_\Lambda(x) \coloneqq \Lambda \braket{x}^{2\gamma - 2} - \braket{x}^\gamma$ uniformly from above, which is possible because $\gamma < 2$.
  Writing
  \begin{equation*}
    f_\Lambda(x) = \braket{x}^\gamma(\Lambda \braket{x}^{\gamma - 2} - 1),
  \end{equation*}
  we see that $f_\Lambda(x) \leq 0$ when $\braket{x} \geq \Lambda^{\frac{1}{2 - \gamma}}$.
  Hence
  \begin{equation*}
    f_\Lambda(x) \leq \Lambda \braket{x}^{2\gamma - 2}|_{\braket{x} = \Lambda^{1/(2 - \gamma)}} = \Lambda^{1 + \frac{2\gamma - 2}{2 - \gamma}} \leq \Lambda^{\frac{4}{2 - \gamma}}.
  \end{equation*}
  Applying this to \eqref{eq:gamma-super}, we see that
  \begin{equation*}
    K \dot a \braket{x}^\gamma - 4 K^2 \braket{x}^{2\gamma - 2} = - K \dot a f_{4K/\dot a}(x) \geq - K \left(\frac{4K}{\dot a}\right)^{\frac{4}{2 - \gamma}} = -K\big[4K(t + 1)^2\big]^{\frac{4}{2 - \gamma}}.
  \end{equation*}
  Thus \eqref{eq:gamma-super} yields
  \begin{equation*}
    \op{NL}[\bar{u}_\gamma] \geq \dot b - 2K b - K\big[4K(t + 1)^2\big]^{\frac{4}{2 - \gamma}}.
  \end{equation*}
  Now, there exists $D(K, \gamma) > 0$ such that $D \e^{Kt} \geq K\big[4K(t + 1)^2\big]^{\frac{4}{2 - \gamma}}$ for all $t \geq 0$.
  It follows that if we choose $b(t) \coloneqq D \e^{3K t}$, then we will have $\op{NH}[\bar u_\gamma] \geq 0$.
  That is, $\bar{u}_\gamma$ is a supersolution of \eqref{eq:main}.

  Now, \eqref{eq:hyp-line} implies that $\ubar{u}_\gamma(0, \anon) \leq u_0 \leq \bar{u}_\gamma(0, \anon)$.
  Again, Lemma~\ref{lem:sandwich} yields a solution $u$ trapped between $\ubar{u}$ and $\bar{u}$; \eqref{eq:line-subquad} follows.
\end{proof}
These power-law bounds imply that $u$ is a strong, root-Lipschitz solution.
\begin{proof}[Proof of Theorem~\textup{\ref{thm:root-Lip}} on the line]
  Let $u$ be a solution from Lemma~\ref{lem:line}.
  As a classical solution of \eqref{eq:main}, $u$ is also a weak solution.
  We now control $\partial_x u$ and $\partial_x^2 u$ to show that $u$ is strong and root-Lipschitz.
  We note that this will imply the uniqueness of $u$ in Lemma~\ref{lem:line}.

  We define a change of variables $\zeta$ depending on the parameter $\gamma \in [0, 2]$.
  If $\gamma = 2$, fix $T \in (0, \kappa^{-1})$ and let $\zeta(y) \coloneqq \sinh y$.
  If $\gamma \in [0, 2)$, fix $T \in (0, \infty)$ and define $\zeta \colon \R \to \R$ by $\zeta(y) \coloneqq \int_0^y \braket{z}^{\frac{\gamma}{2 - \gamma}} \d z.$
  
  In each case, we use $\zeta$ in \eqref{eq:transform} to define $v$.
  One can readily check that the ratios $\zeta''/\zeta'$ and $\zeta'''/\zeta'$ are uniformly smooth, so \eqref{eq:transform-ev} satisfies Hypothesis~\ref{hyp:transformed}.
  Moreover,
  \begin{equation*}
    \zeta' \asymp \braket{\zeta}.
  \end{equation*}
  Using \eqref{eq:CoV0}--\eqref{eq:CoV2}, this implies the following equivalences, each of which holds provided either side of the preceding ones does:
  \begin{align}
    u \asymp \braket{x}^\gamma &\iff v \asymp 1,\label{eq:line-equiv0}\\
    |\partial_x u| \lesssim \braket{x}^{\gamma/2} &\iff |\partial_y v| \lesssim 1,\label{eq:line-equiv1}\\
    |\partial_x^2 u| \lesssim 1 &\iff |\partial_y^2 v| \lesssim 1.\label{eq:line-equiv2}
  \end{align}
  By \eqref{eq:line-equiv0}, Lemma~\ref{lem:line} implies that $v \asymp 1$ on $[0, T] \times \R$.
  Moreover, the root-Lipschitz assumption in Hypothesis~\ref{hyp:init} yields $|\partial_x u_0| \lesssim \sqrt{u_0} \lesssim \braket{x}^{\gamma/2}$, so by \eqref{eq:line-equiv1} we have $|\partial_y v_0| \lesssim 1$.
  Thus $v$ satisfies the hypotheses of Proposition~\ref{prop:transformed}, and $|\partial_y v| \lesssim 1$ and $|\partial_y^2 v| \lesssim t^{-1/2}$.
  Using \eqref{eq:line-equiv1} and \eqref{eq:line-equiv2}, these estimates become
  \begin{equation}
    \label{eq:line-deriv-bd}
    |\partial_x u| \lesssim \braket{x}^{\gamma/2} \And |\partial_x^2 u| \lesssim t^{-1/2} \quad \text{on } (0, T] \times \R.
  \end{equation}
  Hence $u$ is a strong solution on the time interval $[0, \kappa^{-1})$ or $[0, \infty)$, depending on~$\gamma$.
  Moreover, \eqref{eq:sqrt-deriv}, Lemma~\ref{lem:line}, and \eqref{eq:line-deriv-bd} imply that $u$ is root-Lipschitz.
\end{proof}

\subsection{The half-line}

Finally, we treat the delicate case $I = \R_+$.
The half-line encompasses both challenges treated in the previous two subsections: vanishing diffusivity near $0$ and potential blow-up at infinity.
Fortunately, we have already done most of the work.

Because \eqref{eq:main} is degenerate at the boundary $x = 0$ of $I$, we again employ the solutions $u^\eps$ with initial data $u_0 + \eps$.
Their limit $u \coloneqq \lim_{\eps \searrow 0} u^\eps$ is our putative solution.
\begin{lemma}
  \label{lem:half}
  Suppose $u_0$ satisfies Hypothesis~\textup{\ref{hyp:init}}, and in particular \eqref{eq:hyp-half}.
  Then there exists a limit $u = \lim_{\eps \searrow 0} u^\eps$ of a discrete sequence of solutions $(u^\eps)_\eps$ such that
  \begin{equation}
    \label{eq:half-upper-quad}
    u(t, x) \leq \frac{x^2}{\kappa^{-1} - t} \quad \text{on } [0, \kappa^{-1}) \times \R_+.
  \end{equation}
  If $\gamma < 2$, we also have
  \begin{equation}
    \label{eq:half-upper-subquad}
    u(t, x) \leq 2K \braket{x}^\gamma + D \e^{3Kt} \quad \text{on } [0, \kappa^{-1}) \times \R_+,
  \end{equation}
  where $D$ is the constant from Lemma~\ref{lem:line}.
  Moreover, for all $\gamma \in [0, 2]$ we have
  \begin{equation}
    \label{eq:half-lower}
    u(t, x) \geq \frac{x^2 \wedge x^\gamma}{2(\kappa^{-1} + K)} \quad \text{on } [0, \kappa^{-1}) \times \R_+.
  \end{equation}
\end{lemma}
\begin{proof}
  Drawing on the proof of Lemma~\ref{lem:int}, let
  \begin{equation*}
    \bar{u}_2^\eps(t, x) \coloneqq \frac{x^2 + \kappa^{-1} \eps}{\kappa^{-1} - t},
  \end{equation*}
  which is a supersolution of \eqref{eq:main}.
  Recalling $\bar{u}_\gamma$ from the proof of Lemma~\ref{lem:line}, their minimum $\bar{u}^\eps \coloneqq \bar{u}_2^\eps \wedge \bar{u}_\gamma$ is also a (generalized) supersolution of \eqref{eq:main}.
  Moreover, $\bar{u}^\eps(0, \anon) \geq u_0 + \eps$ by \eqref{eq:hyp-half}.

  For the lower bound, take $G \in W_{\loc}^{2,\infty}\big([0, \infty)\big)$ such that $G'' \in L^\infty(\R_+)$ and
  \begin{equation*}
    x^2 \wedge x^\gamma \leq G(x) \leq 2 (x^2 \wedge x^\gamma).
  \end{equation*}
  Using piecewise quadratics, we can arrange $\abs{G''} \leq 4$.
  By \eqref{eq:hyp-half}, $u_0 \geq G/(2K)$.
  We define
  \begin{equation*}
    \ubar{u}^\eps(t, x) \coloneqq \frac{G(x) + \eps}{2(t + K)}.
  \end{equation*}
  Then $\ubar{u}^\eps$ is a subsolution of \eqref{eq:main} satisfying $\ubar{u}^\eps(0, \anon) \leq u_0 + \eps$.

  For each $\eps > 0$, we have constructed a pair of sub- and supersolutions $\ubar{u}^\eps$ and $\bar{u}^\eps$ such that $\ubar{u}^\eps(0, \anon) \leq u_0 + \eps \leq \bar{u}^\eps(0, \anon)$.
  Using Lemma~\ref{lem:sandwich}, we obtain a solution $u^\eps$ of \eqref{eq:main} with initial data $u_0 + \eps$ such that $\ubar{u}^\eps \leq u^\eps \leq \bar{u}^\eps$.
  We have not shown the uniqueness of $u^\eps$, so we do not know that the sequence $(u^\eps)_{\eps}$ is increasing in $\eps$.
  In fact, one can use a suitable change of variables and the comparison principle to prove both.
  We do not require these facts, so we use a simpler strategy here.

  Let $\eps_n \coloneqq \tfrac{1}{n}$.
  We construct a solution $u^{\eps_1}$ as above.
  For $n > 1$, we use $\bar{u}^{\eps_n} \wedge u^{\eps_{n-1}}$ as a supersolution in the construction of $u^{\eps_n}$.
  Thus $u^{\eps_n} \leq u^{\eps_{n-1}}$ for all $n > 1$.
  In this manner, we construct a sequence of solutions $u^{\eps_n}$ with initial data $u_0 + \eps_n$ that is decreasing in $n$.
  This allows us to extract a pointwise limit $u \coloneqq \lim_{n \to \infty} u^{\eps_n}$.

  By construction,
  \begin{equation*}
    \lim_{\eps \searrow 0} \ubar{u}^\eps \leq u \leq \lim_{\eps \searrow 0} \bar{u}^\eps.
  \end{equation*}
  This implies \eqref{eq:half-upper-quad}--\eqref{eq:half-lower}.
\end{proof}
\begin{proof}[Proof of Theorem~\textup{\ref{thm:root-Lip}} on the half-line]
  Let $u$ be a limit of solutions from Lemma~\ref{lem:half}.
  As noted earlier, \cite{DPL87} shows that $u$ is a weak solution of \eqref{eq:main}.
  (Strictly speaking, \cite{DPL87} only treats bounded domains, but a routine cutoff argument allows us to extend the result to $I = \R_+$.)
  We now show that $u$ is strong and root-Lipschitz; this will imply its uniqueness.

  We again define a $\gamma$-dependent change of variables.
  If $\gamma = 2$, let $\zeta(y) \coloneq \e^y$.
  If $\gamma \in [0, 2),$ let $\psi \in \m{C}_\cc^\infty(\R)$ satisfy $0 \leq \psi \leq 1$, $\psi|_{(-\infty, -1]} \equiv 1$, and $\psi|_{[1, \infty)} \equiv 0$.
  Then we define $\zeta \colon \R \to \R_+$ by
  \begin{equation*}
    \zeta(y) \coloneqq \int_{-\infty}^y \left(\psi(z) \e^z + [1 - \psi(z)] \braket{z}^{\frac{2}{2 - \gamma}}\right) \d z.
  \end{equation*}

  In each case, we use $\zeta$ in \eqref{eq:transform} to define $v$.
  Again, the ratios $\zeta''/\zeta'$ and $\zeta'''/\zeta'$ are uniformly smooth, so \eqref{eq:transform-ev} satisfies Hypothesis~\ref{hyp:transformed}.
  Moreover,
  \begin{equation*}
    \zeta' \asymp \zeta \wedge \zeta^{\gamma/2}.
  \end{equation*}
  Using \eqref{eq:CoV0}--\eqref{eq:CoV2}, this implies the following sequential equivalences, each of which holds provided either side of the preceding ones does:
  \begin{align}
    u \asymp x^2 \wedge x^\gamma &\iff v \asymp 1,\label{eq:half-equiv0}\\
    |\partial_x u| \lesssim x \wedge x^{\gamma/2} &\iff |\partial_y v| \lesssim 1,\label{eq:half-equiv1}\\
    |\partial_x^2 u| \lesssim 1 &\iff |\partial_y^2 v| \lesssim 1.\label{eq:half-equiv2}
  \end{align}
  Fix $T \in (0, \kappa^{-1})$.
  By \eqref{eq:half-equiv0}, Lemma~\ref{lem:half} implies that $v \asymp 1$ on $[0, T] \times \R$.
  Moreover, the root-Lipschitz assumption in Hypothesis~\ref{hyp:init} yields $\abs{\partial_x u_0} \lesssim \sqrt{u_0} \lesssim x \wedge x^{\gamma/2}$, so by \eqref{eq:half-equiv1} we have $|\partial_y v_0| \lesssim 1$.
  Thus $v$ satisfies the hypotheses of Proposition~\ref{prop:transformed}, and $|\partial_y v| \lesssim 1$ and $|\partial_y^2 v| \lesssim t^{-1/2}$.
  Using \eqref{eq:half-equiv1} and \eqref{eq:half-equiv2}, these estimates become
  \begin{equation}
    \label{eq:half-deriv-bd}
    |\partial_x u| \lesssim x \wedge x^{\gamma/2} \And |\partial_x^2 u| \lesssim t^{-1/2} \quad \text{on } (0, T] \times \R_+.
  \end{equation}
  Hence $u$ is a strong solution on the time interval $[0, \kappa^{-1})$.
  Also, \eqref{eq:sqrt-deriv}, Lemma~\ref{lem:half}, and \eqref{eq:half-deriv-bd} imply that $u$ is root-Lipschitz.
  This completes the proof of Theorem~\ref{thm:root-Lip}.
\end{proof}

\section{Early blow-up}
\label{sec:blow-up}

To close the paper, we show that strong solutions of \eqref{eq:main} can blow up at zero sooner than one might na\"ively expect.
\begin{proof}[Proof of Theorem~\textup{\ref{thm:blow-up}}]
  Let $u_0 = \kappa x^2 + \vartheta(x)$ with $\vartheta \in \m{C}_\cc^\infty\big(\R_+; [0, \infty)\big)$.
  Let $T_*(\kappa, \vartheta)$ denote the maximal existence time of a strong solution with initial data $u_0$.
  Because $u_0$ satisfies Hypothesis~{\renewcommand\gammabind{2}\renewcommand\kappabind{\bar\kappa}\ref{hyp:init}} for some $\bar\kappa \geq \kappa$ depending on $\vartheta$, Theorem~\ref{thm:root-Lip} implies that $T_* > 0$.
  So there exists a unique strong solution $u$ on the time interval $[0, T_*)$.
  We show that for some $\vartheta,$ we have $T_* < \kappa^{-1}$.

  Given $T \in (0, T_*]$, define the nondecreasing quantity
  \begin{equation*}
    \m{Q}(T) \coloneqq \sup_{[0, T) \times \R_+} x^{-2}u(t, x),
  \end{equation*}
  noting that
  \begin{equation}
    \label{eq:quad-root-Lip}
    \sqrt{\m{Q}(T)} \leq \sup_{t \in [0, T)} \op{Lip}\sqrt{u(t, \anon)}.
  \end{equation}
  By \eqref{eq:quad-root-Lip}, the root-Lipschitz property in Theorem~\ref{thm:root-Lip} ensures that $\m{Q}(T) < \infty$ for all $T < \bar{\kappa}^{-1}$.
  Then Corollary~\ref{cor:strong-implies-RL} allows us to conclude that $\m{Q}(T) < \infty$ for all $T \in (0, T_*)$.

  Conversely, the proof of Theorem~\ref{thm:root-Lip} on the half-line shows that if $Q(T_*) < \infty$, then there exists $\kappa' \in (0, \infty)$ such that $u(T, \anon)$ satisfies Hypothesis~{\renewcommand\gammabind{2}\renewcommand\kappabind{\kappa'}\ref{hyp:init}} for all $T \in (0, T_*).$
  Then Theorem~\ref{thm:root-Lip} would provide a strong extension of $u$ past time $T_*$, contradicting its definition.
  So in fact we must have $\m{Q}(T_*) = \infty$, and the blow-up of $\m{Q}$ is equivalent to the breakdown of a strong solution.
  It now suffices to show that for some choice of $\vartheta$, $\m{Q}$ diverges before time $\kappa^{-1}$.

  To do so, we employ the change of variables $v(t, y) \coloneqq \e^{-2y} u(t, \e^y) - (\kappa^{-1} - t)^{-1}$ for $y \in \R$.
  Then $v(0, y) = \theta(y) \coloneqq \e^{-2y} \vartheta(\e^y) \in \m{C}_\cc^\infty(\R)$.
  Let $a(t) \coloneqq (\kappa^{-1} - t)^{-1}$.
  Our modulated solution $v$ satisfies a variation on \eqref{eq:transform-ev}:
  \begin{equation}
    \label{eq:log-reprise}
    \partial_t v = \frac{1}{2} (v + a) \partial_x^2 v + \frac{3}{2} (v + a) \partial_x v + (v + 2a)v, \quad v(0, \anon) = \theta.
  \end{equation}
  We show that $v$ can blow up before $a$ does---that is, before time $\kappa^{-1}$.
  Our argument takes inspiration from Itaya, who investigated blow-up in parabolic equations with nonlinear diffusivity~\cite{I79}.

  We can check that $a(t)\e^{-y}$ is a (super)solution of \eqref{eq:log-reprise} and $0$ is evidently a (sub)solution.
  Assume that $0 \leq \theta(y) \leq a(0)\e^{-y} = \kappa \e^{-y}$.
  Then the comparison principle yields
  \begin{equation}
    \label{eq:exp-moment}
    0 \leq v(t, y) \leq a(t)\e^{-y} \quad \text{on } [0, T_*) \times \R.
  \end{equation}
  Also, $v$ is uniformly bounded by $\m{Q}(T) < \infty$ on $[0, T] \times \R$ for all $T < T_*$.
  
  Let $\varphi \coloneqq \log(1 + v/a)$, so that \eqref{eq:log-reprise} yields
  \begin{equation*}
    \partial_t \varphi = \frac{1}{2} \partial_x^2 v + \frac{3}{2} \partial_x v + v.
  \end{equation*}
  Hence for any smooth function $\rho \colon \R \to \R_+$ satisfying suitable bounds at infinity,
  \begin{equation*}
    \der{}{t} \int_{\R} \rho\varphi \d y = \frac{1}{2} \int_{\R} \left(\rho'' - 3 \rho' + 2 \rho\right) v \d y.
  \end{equation*}
  By \eqref{eq:exp-moment} and boundedness, we are free to take $\rho(y) = \lambda \e^{\lambda y}$ for any $\lambda \in (0, 1)$.
  Then
  \begin{equation}
    \label{eq:weight}
    \der{}{t} \int_{\R} \rho\varphi = \frac{1}{2}(\lambda^2 - 3 \lambda + 2) \int_{\R} \rho v.
  \end{equation}
  We further assume that $\lambda \in (0, \tfrac{1}{2}]$.
  Using the exponential bound \eqref{eq:exp-moment}, we see that
  \begin{equation*}
    \int_{\R_+} \rho \varphi \leq \lambda\int_{\R_+} \e^{y/2} \log(1 + \e^{-y}) \d y < \lambda \int_{\R_+} \e^{-y/2} \d y = 2\lambda.
  \end{equation*}
  We now integrate \eqref{eq:weight} in time and divide the spatial integrals between $\R_-$ and $\R_+$.
  Dropping one term and using the fact that $\lambda^2 - 3\lambda + 2 > \tfrac{1}{2}$ for $\lambda \in (0, \tfrac{1}{2}]$, we find
  \begin{equation}
    \label{eq:left-mass}
    \int_{\R_-} \rho \varphi(t, \anon) \geq \int_\R \rho \varphi(0, \anon) - 2\lambda + \frac{1}{4}\int_0^t \!\!\!\int_{\R_-} \rho v.
  \end{equation}
  By construction, $\rho$ has unit mass on $\R_-$.
  Thus by Jensen's inequality,
  \begin{equation}
    \label{eq:Jensen}
    \int_{\R_-} \rho \varphi = \int_{\R_-} \rho \log(1 + v/a) \leq \log\left(1 + \frac{1}{a}\int_{\R_-} \rho v\right).
  \end{equation}
  Let $J$ denote the right side of \eqref{eq:Jensen}.
  Combining \eqref{eq:left-mass} and \eqref{eq:Jensen}, we can write
  \begin{equation*}
    J(t) \geq \int_{\R} \rho \log(\kappa^{-1}\theta + 1) - 2\lambda + \frac{1}{4}\int_0^t a(s)\big[\e^{J(s)} - 1\big] \d s.
  \end{equation*}
  Let $b \coloneqq \int_{\R} \rho \log(\kappa^{-1}\theta + 1) - 2 \lambda$.
  Using $a \geq \kappa$ and $J \geq 0$, we obtain
  \begin{equation*}
    J(t) \geq b + \frac{\kappa}{4}\int_0^t \big[\e^{J(s)} - 1\big] \d s \ForAll t \in [0, T_*).
  \end{equation*}
  Let $V(t) \coloneqq \tfrac{\kappa}{4}\int_0^t \big[\e^{J(s)} - 1\big] \d s \geq 0$, so that $J \geq V + b$.
  Differentiating, we find
  \begin{equation}
    \label{eq:V-ev}
    \dot{V} = \frac{\kappa}{4}(\e^J - 1) \geq \frac{\kappa}{4}(\e^{V + b} - 1).
  \end{equation}
  Suppose $b \geq \log 2$.
  Then $\e^{V + b} \geq 2$, so $\e^{V + b} - 1 \geq \tfrac{1}{2} \e^{V + b}$.
  Hence \eqref{eq:V-ev} yields
  \begin{equation*}
    \dot{V} \geq \frac{\kappa}{8} \e^{V + b}.
  \end{equation*}
  Integrating this differential inequality and using $V(0) = 0$, we obtain
  \begin{equation*}
    \e^{V(t)} \geq \frac{8}{8 - \kappa \e^b t}.
  \end{equation*}
  This expression blows up at time $8 \kappa^{-1}\e^{-b}$, so $J$ blows up at or before this time.
  Because $\rho$ is integrable on $\R_-$, we must have $\sup_{\R_-} v \to \infty$ and hence $\m{Q} \to \infty$ if $J \to \infty$.
  Because $\m{Q}$ blows up precisely at time $T_*,$ it therefore follows (still under the assumption that $b\ge\log 2$) that
  \begin{equation}
    \label{eq:blow-up}
    T_* \leq \frac{8}{\kappa \e^b} = \frac{8}{\kappa} \exp\left(2\lambda - \int_{\R} \rho \log(\kappa^{-1}\theta + 1)\right).
  \end{equation}
  Recall that our only constraint on the bump function $\theta$ is $0 \leq \theta(y) \leq \kappa\e^{-y}$.
  Taking a sequence of such $\theta$s that locally converge to $\kappa \e^{-y}$, we see that
  \begin{equation*}
    \sup_{\text{valid } \theta} \int_{\R} \rho \log(\kappa^{-1}\theta + 1) \geq \lambda \int_{\R_-} \e^{\lambda y} \log(\e^{-y} + 1) \d y > \lambda \int_{\R_-} \e^{\lambda y} \abs{y} \d y = \frac{1}{\lambda}.
  \end{equation*}
  So there exists $\theta_\lambda \in \m{C}_\cc^\infty(\R)$ such that $0 \leq \theta_\lambda \leq \kappa\e^{-y}$ and $\int_{\R} \rho \log(\kappa^{-1}\theta_\lambda + 1) = \lambda^{-1}$.
  With this initial condition, we have
  \begin{equation*}
    b_\lambda \coloneqq \int_{\R} \rho \log(\kappa^{-1}\theta_\lambda + 1) - 2\lambda = \frac{1}{\lambda} - 2 \lambda.
  \end{equation*}
  We can therefore choose $\lambda \in (0, \tfrac{1}{2}]$ sufficiently close to $0$ that $b_\lambda \geq \log 16 > \log 2.$
  Then \eqref{eq:blow-up} implies $T_* \leq \tfrac{1}{2\kappa}$.
  Moreover, because $\m{Q}(T_*) = \infty$, \eqref{eq:quad-root-Lip} implies \eqref{eq:not-root-Lip}.
\end{proof}

\printbibliography
\end{document}